\documentclass[10pt]{amsart}
\usepackage{amstext,amsfonts,amssymb,amscd,amsbsy,amsmath}
\usepackage{enumerate}

\usepackage[all]{xy}
\SelectTips{eu}{}
\xyoption{curve}
\CompileMatrices

\usepackage[colorlinks]{hyperref}

\usepackage[all]{xy}
\SelectTips{eu}{}

\newcommand{\bx}{\ensuremath{\boldsymbol{x}}}

\newcommand{\fm}{\ensuremath{\mathfrak m}}

\newcommand{\fp}{\ensuremath{\mathfrak p}}
\newcommand{\fq}{\ensuremath{\mathfrak q}}

\newcommand{\ft}{\ensuremath{\mathfrak t}}

\newcommand{\ov}{\overline}
\newcommand{\ges}{{\geqslant}}
\newcommand{\les}{{\leqslant}}
\newcommand{\col}{{\colon}}

\newcommand{\BZ}{\ensuremath{\mathbb Z}}

\DeclareMathOperator{\HH}{H}
\DeclareMathOperator{\Ker}{Ker}

\DeclareMathOperator{\rank}{rank}

\DeclareMathOperator{\depth}{depth}
\DeclareMathOperator{\height}{height}

\DeclareMathOperator{\edim}{edim}

\DeclareMathOperator{\shift}{\mathsf{\Sigma}}

\newcommand{\ext}[4]{\operatorname{Ext}^{#1}_{#2}(#3,#4)}
\newcommand{\Tor}[4]{\operatorname{Tor}_{#1}^{#2}(#3,#4)}

\newcommand{\Hom}[3]{\operatorname{Hom}_{#1}(#2,#3)}

\newcommand{\Po}[2]{\ensuremath{P^{#1}_{#2}(t)}}
\newcommand{\ba}[1]{\ensuremath{I_{#1}^{#1}(t)}}
\newcommand{\baa}[2]{\ensuremath{I_{#1}^{#2}}}

\newcommand{\dd}{\ensuremath{\partial}}
\newcommand{\lra}{\ensuremath{\longrightarrow}}

\newcommand{\xra}{\ensuremath{\xrightarrow}}
\newcommand{\xla}{\ensuremath{\xleftarrow}}

\theoremstyle{plain}
\newtheorem{theorem}{Theorem}[section]

\newtheorem{corollary}[theorem]{Corollary}
\newtheorem{lemma}[theorem]{Lemma}

\theoremstyle{definition}

\newtheorem{subsec}[theorem]{}
\newtheorem{example}[theorem]{Example}

\newtheorem*{Subsec}{}

\newtheorem{subsubsec}{}
\numberwithin{subsubsec}{theorem}

\theoremstyle{remark}
\newtheorem{remark}[theorem]{Remark}
\newtheorem{conjecture}[theorem]{Conjecture}
\newtheorem{question}[theorem]{Question}
\newtheorem{problems}[theorem]{Questions}
\newtheorem*{Remark}{Remark}

\numberwithin{equation}{theorem}

\begin{document}

\title[Local rings of embedding codepth $3$]
      {A cohomological study of local rings \\
      of embedding codepth $3$}

\author[L.~L.~Avramov]{Luchezar L.~Avramov}
\address{Department of Mathematics,
   University of Nebraska, Lincoln, NE 68588, U.S.A.}
     \email{avramov@math.unl.edu}

\thanks{The author was partly supported by NSF grants DMS 0803082 and DMS 1103176}


  \date{\today}
  

  \begin{abstract}
The generating series of the Bass numbers $\mu^i_R=\rank_k\ext iRkR$ of local rings $R$ 
with residue field $k$ are computed in closed rational form, in case the embedding dimension 
$e$ of $R$ and its depth $d$ satisfy  $e-d\le3$.  For each such $R$ it is proved that there is 
a real number $\gamma>1$, such that $\mu^{d+i}_R\ge\gamma\mu^{d+i-1}_R$ holds for all 
$i\ge 0$, except for $i=2$ in two explicitly described cases, where $\mu^{d+2}_R=\mu^{d+1}_R=2$.  
New restrictions are obtained on the multiplicative structures of minimal 
free resolutions of length $3$ over regular local rings.  
  \end{abstract}

\maketitle

\section*{Introduction}

The paper concerns cohomological invariants of commutative noetherian local rings.
Let $R$ be such a ring, $\fm$ its maximal ideal, and let $d$ denote the depth of~$R$ 
and $e$ the minimal number of generators of $\fm$.  The number $e-d$ is called
the \emph{embedding codepth} of $R$.  It is equal to the length of a minimal free 
resolution $F$ of $\widehat R$ over $P$, where $\widehat R$ is the $\fm$-adic 
completion of $R$ and $P$ a regular local ring of dimension $e$, for which there 
is an isomorphism $\widehat R\cong P/I$; 
such an isomorphism always exists, due to Cohen's Structure Theorem.  For 
$c\le2$ the structure of $F$, and hence that of $\widehat R$, is determined 
by the Hilbert-Burch Theorem.

This paper is mostly concerned with rings of codepth $3$, so we assume $c=3$
for the rest of the introduction.   There exist then integers $l\ge2$ and $n\ge1$,
such that
  \[
      \tag{1}
F=\quad
0\longrightarrow P^n\xra{\ \dd_3\ } P^{n+l}\xra{\ \dd_2\ } 
P^{l+1}\xra{\ \dd_1\ } P\longrightarrow 0
  \]
The maps $\dd_i$ are known in a few cases only.  Buchsbaum and Eisenbud 
described them in \cite{BE1} for $l=2$, and in~\cite{BE2} when $R$ is 
Cohen-Macaulay with $l=3$ or $n=1$.  A.~Brown determined $\dd_i$ for 
certain Cohen-Macaulay rings with $n=2$; see~\cite{Br}.

The proofs of those theorems use the fact that $F$ can be turned into a 
graded-commutative DG (that is, differential graded) algebra; see~\cite{BE2}.  
Such a structure is not unique in general, but the isomorphism class of the 
graded $k$-algebra 
  \[
      \tag{2}
A=F\otimes_Pk\,,
\quad\text{where}\quad
k=R/\fm\,,
  \]
is an invariant of~$R$.  The possible isomorphism classes were determined by 
Weyman~\cite{We} in characteristic zero and by Avramov, Kustin, and Miller 
\cite{AKM} in general.  The remarkable fact is that for fixed $l$ and $n$ there 
exist only finitely many possibilities for $A$, described explicitly by simple 
multiplication tables.  These are reviewed in Section \ref{S:Background},
along with other background material.

We are interested in classifying non-Gorenstein rings.  A natural tool for the
task is provided by the \emph{Bass numbers} $\mu^i_R=\rank_k\ext iRkR$,
which are positive for all $i>d$ when $R$ is not Gorenstein, but vanish when 
it is.  The \emph{Bass series} $\ba R=\sum_{i\ges0}\mu^i_R\,t^i$ 
offers a useful format for recording the Bass numbers of $R$.

As our first result, Theorem \ref{thm:series}, we obtain in closed form expressions
  \[
      \tag{3}
\ba R=\frac{f(t)}{g(t)}
\quad\text{and}\quad
\Po Rk=\frac{(1+t)^{e-1}}{g(t)}
\quad\text{with}\quad 
f(t),g(t)\in\BZ[t]\,,
  \]
where $\Po Rk=\sum_{i\ges0}\rank_k\Tor iRkk\,t^i$ is the \emph{Poincar\'e
series} of $k$.  That such expressions \emph{exist} follows from \cite{AKM}, 
via \cite{Fo2}, and $g(t)$ was computed in \cite{Av:msri}.  

For the goals of this paper we need the precise form of $f(t)$ as well.  
In Section~\ref{S:Series} the series $\ba R$ and $\Po Rk$ are computed in parallel.
Work in \cite{AFL} and~\cite{Av:small} reduces the problem to finding $\ba A$ and 
$\Po Ak$ for the algebra $A$ in (2).  To compute these series we use a battery of change-of-rings 
results, which are analogs of known theorems over local rings.  Translation to the 
context of graded-commutative $k$-algebras requires changes in statements and 
proofs; these are discussed in Appendix~\ref{S:Graded algebras}.

It has long been known that that for Gorenstein rings $l$ is even, see \cite{Wa}, and that 
$R$ is Gorenstein if and only if $A$ has Poincar\'e duality, see \cite{AG}, so $n=1$.  Furthermore,
$R$ is complete intersection if and only if $A$ is an exterior algebra, see \cite{As}, 
and then $l=c-1$.  In Theorem \ref{thm:class} we prove that membership in each 
one of the remaining classes imposes new restrictions on the numbers $l$ 
and $n$.  The arguments introduce ideas that have not
been applied earlier in this context, such as utilizing the DG module structure of 
$\Hom PFP$ over the DG algebra $F$ from (1), and analyzing the growth of the 
Betti numbers of $\widehat R$ over complete intersection quotient rings of~$P$.  

In the first three sections the focus is on the structure of rings of codepth $3$.  The 
last section is motivated by open problems on the behavior of Bass sequences of local 
rings in general.  In the introduction of \cite{CSV}, Christensen, Striuli, and Veliche collect 
precise questions and give a comprehensive survey of earlier results.  

Theorem \ref{thm:growth} gives complete answers in codepth $3$: When $R$ 
is not Gorenstein
  \[
      \tag{4}
\mu_R^{d+i}\ge\gamma\mu_R^{d+i-1}
  \]
holds for some real number $\gamma>1$ and every integer $i\ge1$, with a single exception:
  \[
      \tag{5}
\mu^{d+2}_R=\mu^{d+1}_R=2\quad\text{when}\quad \widehat R\cong P/(wx,wy,z)
  \]
and $w$ is $P$-regular, $x,y$ is a $P$-regular sequence, and $z$ is $P/(wx,wy)$-regular.   
In particular, we recover the \emph{asymptotic} information known from earlier work: The 
Bass sequence of $R$ eventually is either constant or grows exponentially, 
see \cite{Av:msri}; it is unbounded when $R$ is Cohen-Macaulay, but not Gorenstein, see 
Jorgensen and Leuschke~\cite{JL}; if it is unbounded, then (4) holds for $i\gg0$, 
see~Sun \cite{Su}.

Neither the inequalities in (4), nor the description of the exceptions in (5), are formal 
consequences of the rational expressions in (3).  In fact, extracting information on the 
Taylor coefficients of a rational function from expressions for its numerator and 
denominator is classically known to be a very hard problem.  

Our approach is to prove first that $\mu_R^{d+i}>\mu_R^{d+i-1}$ holds, with the 
exceptions in (5), by drawing on three distinct sources---the expressions of the 
coefficients of $f(t)$ and $g(t)$ from Theorem \ref{thm:series}, the relations between 
those coefficients implied by Theorem~\ref{thm:class}, and certain growth properties 
of the Betti numbers of $k$ that are satisfied whenever $R$ is not complete intersection.  
Once the growth of the Bass sequence is established, Theorem~\ref{thm:growth} 
easily follows from results in \cite{Av:msri} and \cite{Su}.   

Since no additional effort is involved, all the results in the paper are stated and proved
for local rings of embedding codepth \emph{at most} $3$.

\section{Background}
    \label{S:Background}

In this paper we say that $(R,\fm,k)$ is a local ring if $R$ is a commutative 
noetherian ring, $\fm$ its unique maximal ideal, and $k=R/\fm$.  Recall the
invariants
  \[
\edim R=\rank_k(\fm/\fm^2)
\quad\text{and}\quad 
\depth R=\inf\{i\in\BZ\mid\mu^i_R\ne0\}\,.
  \]

  \begin{subsec}
    \label{invariants}
The following notation is fixed for the rest of the paper:
  \[
e=\edim R\,, \quad d=\depth R\,,\quad c=e-d\,,  \quad\text{and}\quad h=\dim R-d\,.
  \]

We write $K$ for the Koszul complex on a minimal set of generators of $\fm$.  It is
a DG algebra over $R$, so its homology is a graded algebra with $\HH_0(K)=k$.  We set
  \[
A=\HH(K)
  \]
and fix notation for the ranks of some $k$-vector spaces associated with $A$:
 \setcounter{equation}{0}
   \begin{equation*}
  \begin{aligned}
l&=\rank_kA_1-1     \qquad & p&=\rank_k(A_1^2)
  \\
m&=\rank_kA_2         & q&=\rank_k(A_1\cdot A_2)
  \\
n&=\rank_kA_3     & r&=\rank_k(\delta_2)
  \end{aligned}
  \end{equation*}
where $\delta_2\col A_2\to\Hom k{A_1}{A_3}$ is defined by $\delta_2(x)(y)=xy$ 
for $x\in A_{2}$ and $y\in A_1$.
  \end{subsec}
 
    \begin{subsec}
    \label{cst}
Let $\widehat R$ denote the $\fm$-adic completion of $R$.  
Cohen's Structure Theorem yields $\widehat R\cong P/I$ for some regular
local ring $(P,\fp,k)$ with $\dim P=e$; that is, $I\subseteq\fp^2$.

When $I$ can be generated by a regular sequence $R$ 
is said to be \emph{complete intersection}; this property is independent
of the choice of presentation, see \cite[2.3.4(a)]{BH}.

Let $F$ be a minimal free resolution of $\widehat R$ over $P$ and 
$L$ be the Koszul complex on a minimal generating set of the maximal
ideal of $P$.   There are natural maps
  \begin{equation}
    \label{eq:chain}
K=R\otimes_RK\xra{\simeq}\widehat R\otimes_RK\xra{\cong}\widehat R\otimes_PL
\xla{\simeq}F\otimes_PL\xra{\simeq}F\otimes_Pk\,;
  \end{equation}
the symbol $\simeq$ denotes a quasi-isomorphism.  In particular, an equality
  \begin{equation}
    \label{eq:switch}
\rank_kA_i=\rank_P F_i
  \end{equation}
holds for every integer $i$.  The Auslander-Buchsbaum Equality now yields 
  \begin{equation}
    \label{eq:ab}
\max\{i\mid A_i\ne0\}=c=\operatorname{pd}_P\widehat R\,.
  \end{equation}

Krull's Principal Ideal Theorem gives the inequalities below; the first equality is \eqref{eq:switch} 
for $i=1$; the third one comes from the catenarity of the regular ring $P$:
  \begin{equation}
     \label{eq:codim}
l+1=\rank_k(I/\fp I)\ge \height_P(I)=e-\dim R=c-h\ge0\,.
  \end{equation}
By definition, the second inequality becomes an equality if and only if $R$ is \emph{regular}.
Since $P$ is Cohen-Macaulay, the first inequality becomes an equality precisely when 
$I$ is generated by a regular sequence; that is, when $R$ is complete intersection.  

When $c\le3$, the equality $\sum_{i\ges 0}(-1)^{i}\rank_PF_i=0$, \eqref{eq:switch}, and 
\eqref{eq:ab} give
  \begin{equation}
    \label{eq:m}
m=l+n\,.
  \end{equation}
    \end{subsec}

The following classification is the starting point for our work and is used throughout 
the paper.  As always, $\bigwedge_k$ denotes the exterior algebra functor.  The functors 
$\shift$ and $\Hom k-{\shift^3k}$ and the construction $\ltimes$ are defined below, in \ref{DG}.

  \begin{subsec}
    \label{table}
If $c\le3$, then up to isomorphism $A$ is described by the following table, 
where $B$, $C$, and $D$ are graded $k$-algebras, and $W$ a graded 
$B$-module with $(B_+)W=0$:

%

\medskip\noindent
\centerline{
 \begin{tabular}{l||l|l|l|l|l}
{Class} \hfill{\hbox{[range]}}
&$c$
& $A$
& $B$
& $C$
& $D$
\\
\hline\hline
$\mathbf{C}(c)$\hfill{\hbox{$[c\ge0]$}}
& $c$
& $B$
& $\textstyle{\bigwedge}_k\shift k^c$
&
&
  \\
$\mathbf{S}$
& $2$
& $B\ltimes W$
& $k$
&
&
  \\
$\mathbf{T}$
& $3$
& $B\ltimes W$
&$C\ltimes\shift(C/C_{\ges 2})$
& $\textstyle{\bigwedge}_k\shift k^2$
&
  \\
$\mathbf{B}$
& $3$
& $B\ltimes W$
& $C\ltimes\shift C_+$
& $\textstyle{\bigwedge}_k\shift k^2$
&
  \\
$\mathbf{G}(r)$\hfill{\hbox{$[r\ge2]$}}
& $3$
& $B\ltimes W$
& $C\ltimes \Hom kC{\shift^3k}$
& $k\ltimes\shift k^r$
&
  \\
$\mathbf{H}(p,q)$ \hfill{\hbox{$[p,q\ge0]$}}
&  $3$
& $B\ltimes W$
& $C\otimes_kD$
& $k\ltimes(\shift k^p\oplus\shift^2 k^q)$\ 
& $k\ltimes\shift k$ 
  \end{tabular}
  }
  
\medskip\noindent
No two algebras $A$ in the table are isomorphic, and neither are any two algebras $B$.  

The table is compiled as follows.  If $c\le1$, then $A_i=0$ for $i>c$
and $A_1\cong k^c$, by \eqref{eq:ab} and \eqref{eq:switch}, whence 
$A\cong{\bigwedge}_k\shift k^c$.  If $c=2$, then $F$ is given by the 
Hilbert-Burch Theorem; an explicit multiplication on $F$, see \cite[2.1.2]{Av:barca}, 
yields $A\cong{\bigwedge}_k\shift k^2$ or $A\cong k\ltimes W$.  When $c=3$ the 
possible isomorphism classes of $A$ are determined in \cite[Proof of 4.1]{We} 
when $k$ has characteristic $0$, and in \cite[2.1]{AKM} in general.
    \end{subsec}

In some cases, the class of a ring and its structure determine each other:

  \begin{subsec}
     \label{euler}
Let $R$ be a local ring with $\edim R-\depth R=c\le3$.

 \begin{subsubsec}
     \label{CI}
The ring $R$ is complete intersection of codimension $c$ if and only if it is in $\mathbf{C}(c)$,
as proved by Assmus \cite[2.7]{As}, see also \cite[2.3.11]{BH}; for such rings $l=c-1$. 
   \end{subsubsec}

 \begin{subsubsec}
     \label{Gorl}
The ring $R$ is Gorenstein, but not complete intersection, if and only if it is in 
$\mathbf{G}(r)$ with $l=r-1$ and $n=1$; for such rings $l$ is even and $l\ge4$.

Indeed, $R$ is Gorenstein if and only if $A$ has Poincar\'e duality, by \cite{AG},
and then $l$ is even, by J.~Watanabe \cite[Thm.]{Wa}; alternatively, see 
\cite[2.1]{BE2} or \cite[3.4.1]{BH}.
  \end{subsubsec}

  \begin{subsubsec}
     \label{golod}
The ring $R$ is Golod if and only if it is in $\mathbf{S}$ or in $\mathbf{H}(0,0)$.

By definition, $R$ is Golod if and only if all Massey products of elements of $A_+$ 
are trivial. The binary ones are just ordinary products.  Massey products of three or 
more elements have degree at least $4$, and $A_i=0$ for $i\ge4$.  Thus, $R$ is 
Golod if and only if $A_+^2=0$.  By \ref{table}, this occurs precisely for the
rings in $\mathbf{S}$ or $\mathbf{H}(0,0)$. 
  \end{subsubsec}
    \end{subsec}

We recall a modicum of notation and facts concerning DG modules.

  \begin{subsec}
     \label{DG}
Let $E$ be a DG algebra over a commutative ring $S$.  We assume
$E_i=0$ for $i<0$ and that $E$ is \emph{graded-commutative}, meaning 
that $xy=(-1)^{ij}yx$ holds for all $x\in E_i$ and $y\in E_j$ and $x^2=0$ 
when $i$ is odd.  The DG algebra $E$ acts on its module $M$ from the left. 
All differentials have degree $-1$.  A \emph{morphism} of DG modules
is a degree zero $E$-linear map that commutes with the differentials;
if it induces isomorphisms in homology in all degrees, it is called 
a \emph{quasi-isomorphism}.

For every $s\in\BZ$, set $(\shift^s M)_j=M_{j-s}$ for $j\in\BZ$.  The identity 
maps on $M_j$ define a bijective map $\varsigma^s\col M\to\shift^sM$ of 
degree $s$.  Setting $\dd^{\shift^s M}(\varsigma(m))=(-1)^s\varsigma^{s}(\dd^M(m))$ 
and $x\varsigma^s(m)=(-1)^{is}\varsigma(xm)$ for every $x\in E_i$ turns  
$\shift^s M$ into a DG $E$-module. 

Recall that $\Hom SM{\shift^sS}$ denotes the DG $E$-module with 
$\Hom SM{\shift^sS}_j=\Hom S{M_{s-j}}S$, differential $\dd(\mu)(m)=(-1)^{j+1}\mu\dd(m)$
for $\mu\in\Hom SM{\shift^sS}_j$, and $E$ acting by 
$(x\mu)(m)=(-1)^{ij}\mu(xm)$ for $x\in E_{i}$.  Set $M^*=\Hom SM{S}$.

The \emph{trivial extension} $E\ltimes M$ is the DG algebra with 
underlying complex $E\oplus M$ and product 
$(x,m)(x',m')=(xx', xm'+(-1)^{ji'}x'm)$ for $x'\in E_{i'}$ and $m\in M_j$.
  \end{subsec}

  \begin{subsec}
     \label{DGD}
Let $E$ be a DG algebra over $S$, and let $M$ and $N$ be DG $E$-modules

Modules $\Tor iEMN$ and $\ext iEMN$ over the ring $S$ are defined for 
every integer $i$, see \cite[\S1]{AH}.  If 
$E$ is a ring, considered as a DG algebra concentrated in degree $0$, and
$M$ and $N$ are $E$-modules, treated as DG modules in a similar way, 
then these derived functors coincide with the classical ones.  When 
$k$ is a field $E\to k$ is a homomorphism of DG algebras, and the 
$k$-vector spaces $\Tor iEMk$ and $\ext iEkN$ have finite rank for each 
$i$ and vanish for $i\ll0$,  we set 
  \begin{align}
    \label{eq:defP}
\Po EN&=\sum_{i\in\BZ}\rank_k\Tor iEMk\, t^i\in\BZ[\![t]\!][t^{-1}]\,.
 \\
    \label{eq:defI}
\baa EN&=\sum_{i\in\BZ}\rank_k\ext iEkN\, t^i\in\BZ[\![t]\!][t^{-1}]\,.
  \end{align}

Every morphism of DG algebras $\varepsilon\col E'\to E$ induces natural 
homomorphisms of $S$-modules $\Tor i{E'}MN\to\Tor iEMN$ 
and $\ext iEMN\to\ext i{E'}MN$ for each $i\in\BZ$.  These maps are 
bijective when $\varepsilon$ is a quasi-isomorphism.   

A \emph{graded algebra} over $S$ is a DG algebra with zero differential; a
\emph{graded module} over a graded algebra is a DG module with zero differential.
  \end{subsec}

  \begin{subsec}
    \label{cst2}
Let $K$ be the Koszul complex $K$ described in \ref{invariants}.
The natural map $K\to k$ turns $k$ into a DG $K$-module.  {From}
\cite[3.2]{Av:small} and \cite[4.1]{AFL}, respectively, we get
  \[
\Po Rk=(1+t)^e\cdot\Po Kk
  \quad\text{and}\quad
\ba R=t^{e}\cdot\ba K\,.
  \]

If the resolution $F$ in \ref{cst} has a structure of DG algebra over $P$, 
then the natural surjection $F_0=P\to k$ turns $k$ into a DG $F$-module.
The maps in \eqref{eq:chain} then are morphisms of DG algebras,
so we get isomorphisms of DG algebras
  \begin{equation}
    \label{eq:DGchain}
F\otimes_Pk=\HH(F\otimes_Pk)\cong A\,.
  \end{equation}
The invariance under quasi-isomorphisms of the DG derived functors in
\ref{DGD} gives 
  \[
\Po Kk=\Po Ak
  \quad\text{and}\quad
\ba K=\ba A\,.
  \]

When $c\le3$ we have $F_i=0$ for $i>3$, see \eqref{eq:ab}, so $F$ supports
a structure of DG algebra over $P$ by \cite[1.3]{BE2}; see also \cite[2.1.4]{Av:barca}.  
Thus, in this case we have
  \begin{align}
    \label{eq:Preduction}
\Po Rk&=(1+t)^e\cdot\Po Ak\,.
  \\
     \label{eq:Breduction}
\ba R&=t^{e}\cdot\ba A\,.
  \end{align}
  
Techniques for computing Poincar\'e series and Bass series over graded
algebras are presented in Appendix \ref{S:Graded algebras}, along with a 
number of examples.
   \end{subsec}

\section{Bass series}
    \label{S:Series}

Our goal in this is section is to prove the following result.	

 \begin{theorem}
    \label{thm:series}
Let $(R,\fm,k)$ be a local ring, set $d=\depth R$ and $e=\edim R$, and let 
$l$, $n$, $p$, $q$, and $r$ be the numbers defined in \emph{\ref{invariants}}.

When $e-d=c\le3$ there are equalities 
  \begin{equation*}
\Po Rk=\frac{(1+t)^{e-1}}{g(t)}
  \quad\text{and}\quad
\ba R=t^d\cdot\frac{f(t)}{g(t)}
  \end{equation*}
where $f(t)$ and $g(t)$ are polynomials in $\BZ[t]$, listed in the
following table: 

\medskip\noindent
\centerline{
   \begin{tabular}{l||l|l}
\emph{Class}\qquad
&$g(t)$
&$f(t)$
\\ 
\hline\hline
&& \\
$\mathbf{C}(c)$
& $(1-t)^c(1+t)^{c-1}$
& $(1-t)^c(1+t)^{c-1}$
\\&& \\
$\mathbf{S}$
& ${1-t-lt^2}$
& ${l+t-t^2}$
  \\ && \\
$\mathbf{T}$
& ${1-t-lt^2-(n-3)t^3-t^5}$
& ${n+lt-2t^2-t^3+t^4}$
  \\ && \\
$\mathbf{B}$
& ${1-t-lt^2-(n-1)t^3+t^4}$
& ${n+(l-2)t-t^2+t^4}$
  \\ && \\
$\mathbf{G}(r)$
& ${1-t-lt^2-nt^3+t^4}$
& ${n+(l-r)t-(r-1)t^2-t^3+t^4}$\qquad
  \\ && \\
$\mathbf{H}(0,0)$
& ${1-t-lt^2-nt^3}$\qquad
& ${n+lt+t^2-t^3}$
  \\ && \\
$\mathbf{H}(p,q)$
& ${1-t-lt^2-(n-p)t^3+qt^4}$\qquad
& ${n+(l-q)t-pt^2-t^3+t^4}$
  \\ 
${p+q\ge1}$ &&
 \end{tabular}
  }
   \end{theorem}

%

All of the Poincar\'e series and a smattering of the Bass series above are known:

  \begin{remark}
Since rings in $\mathbf{C}(c)$ are complete intersection, see \ref{CI}, 
the formula for $\Po Rk$ is due to Tate \cite[Thm.\,6]{Ta}; we have $\ba R=t^d$
because $R$ is Gorenstein. 

When $R$ is in $\mathbf{G}(r)$ with $r=l+1$ and $n=1$, it is Gorenstein 
by \ref{Gorl}.  The formula for $\Po Rk$ then is due to Wiebe \cite[Satz 9]{Wi};
the other formula gives $\ba R=t^d$.

When $R$ is Golod, $\Po Rk$ is given by Golod \cite{Go} and $\ba R$ by 
Avramov and Lescot \cite{AL}.  In view of \ref{golod}, this covers the rings $R$ in 
$\mathbf{S}$ and $\mathbf{H}(0,0)$; for $R$ in $\mathbf{S}$, Scheja \cite[Satz 9]{Sc} 
computed $\Po Rk$ and Wiebe \cite[Satz~8]{Wi} calculated $\ba R$.

The formulas for $\Po Rk$ in the remaining cases were obtained in \cite[3.5]{Av:msri}.
  \end{remark}

In the proof that follows the series $\Po Rk$ and $\ba R$ are computed 
simultaneously and in a uniform manner.  A separate calculation is needed for each class.

\begin{proof}[Proof of Theorem \emph{\ref{thm:series}}]
By \eqref{eq:Preduction} and \eqref{eq:Breduction}, it suffices to establish the equalities
  \begin{equation*}
\frac1{\Po Ak}=(1+t)\cdot{g(t)}
  \quad\text{and}\quad
\frac{\ba A}{\Po Ak}=t^{-c}\cdot(1+t)\cdot{f(t)}\,.
  \end{equation*}  

\begin{Subsec}[Class $\mathbf{C}(c)$]
The formulas come from \eqref{eq:exteriorP} and \eqref{eq:exteriorI}, respectively.
\end{Subsec}

\begin{Subsec}[Class $\mathbf{S}$]
The formulas come from \eqref{eq:nullP} and \eqref{eq:nullI}, respectively.
\end{Subsec}

\begin{Subsec}[Class $\mathbf{T}$]
The exact sequence $0\to\shift^2k\to C\to C/C_{\ges 2}\to0$ and formulas 
\eqref{eq:maximalP} and \eqref{eq:shift} give $\Po C{\shift(C/C_{\ges 2})}=t(1+t^3\Po Ck)$.
Now \eqref{eq:trivialP} and \eqref{eq:exteriorP} yield
 \begin{align*}
\frac1{\Po Bk}
=(1-t^2)^2\bigg(1-t^2\bigg(1+t^3\frac{1}{(1-t^2)^2}\bigg)\bigg)=1-3t^2+3t^4-t^5-t^6\,.
 \end{align*}
The isomorphism of $k$-algebras $B\cong E/E_{\ges 3}$ with 
$E=\textstyle{\bigwedge}_k\shift k^3$ and \eqref{eq:truncatedI} give
  \[
\frac{\ba B}{\Po Bk}=t^{-4}\big(1-(1-3t^2+3t^4-t^5-t^6)\big)-t={3t^{-2}-3+t^2}\,.
 \]
Using formulas \eqref{eq:trivialP} and \eqref{eq:trivialI} we now obtain:
 \begin{align*}
\frac1{\Po Ak}
&=(1-3t^2+3t^4-t^5-t^6)-t\big((l-2)t+(l+n-3)t^2+nt^3\big)
  \\
&=(1+t)\big(1-t-lt^2-(n-3)t^3-t^5\big)\,.
  \\
\frac{\ba A}{\Po Ak}
&=(3t^{-2}-3+t^2)+\big((l-2)t^{-1}+(l+n-3)t^{-2}+nt^{-3}\big)
  \\
&=t^{-3}(1+t)(n+lt-2t^2-t^3+t^4)\,.
 \end{align*}
\end{Subsec}

\begin{Subsec}[Class $\mathbf{B}$]
Using \eqref{eq:shift}, \eqref{eq:maximalP}, and \eqref{eq:exteriorP} we obtain
 \[
\Po C{\shift C_+}
=t\cdot\Po C{C_+}
=t\cdot t^{-1}\cdot(\Po Ck-1)
=\frac{1}{(1-t^2)^2}-1
=\frac{2t^2-t^4}{(1-t^2)^2}\,.
  \]
{From} formulas \eqref{eq:trivialP} and \eqref{eq:exteriorP} we now get
 \begin{align*}
\frac1{\Po Bk} 
&=(1-t^2)^2-t(2t^2-t^4)
=1-2t^2-2t^3+t^4+t^5\,.
 \end{align*}

Assume, for the moment, that there is an exact sequence of graded $B$-modules
  \begin{equation}
    \label{eq:B1}
0\lra\shift^{-2}B_+\oplus\shift^{-1}k\lra\shift^{-2}B\oplus\shift^{-3}B\lra B^*\lra0
  \end{equation}
where $B^*=\Hom kBk$. We then have a string of equalities, where the first one comes from 
\eqref{eq:dual}, and the second one from \eqref{eq:maximalP} and \eqref{eq:shift} applied to 
\eqref{eq:B1}:
   \begin{align*}
\frac{\ba B}{\Po Bk}
&=\frac{\Po B{B^*}}{\Po Bk}
\\
&=\frac{t(t^{-2}\cdot t^{-1}\cdot(\Po Bk-1)+t^{-1}\cdot\Po Bk)+t^{-3}+t^{-2}}{\Po Bk}
\\
&=1+t^{-2}+t^{-3}\cdot\frac{1}{\Po Bk}
  \\
&=t^{-3}+t^{-2}-2t^{-1}-1+t+t^2\,.
 \end{align*}
Formulas \eqref{eq:trivialP} and \eqref{eq:trivialI} now yield:
 \begin{align*}
\frac1{\Po Ak}
&=(1-2t^2-2t^3+t^4+t^5)
-t\big((l-1)t+(l+n-3)t^2+(n-1)t^3\big)
\\
&=(1+t)\big(1-t-lt^2-(n-1)t^3+t^4\big)\,.
\\
\frac{\ba A}{\Po Ak}
&=(t^{-3}+t^{-2}-2t^{-1}-1+t+t^2)
  \\
&{\phantom{{}={}(t^{-3}}}
+\big((l-1)t^{-1}+(l+n-3)t^{-2}+(n-1)t^{-3}\big)
 \\
&={t^{-3}(1+t)\big(n+(l-2)t-t^2+t^4\big)}\,.
 \end{align*}
 
It remains to construct the sequence \eqref{eq:B1}.  To do this we use the
module structures on suspensions and dual modules, described in \ref{DG}.
Recall from \ref{table}~that $B=C\ltimes\shift C_+$, with $C=\bigwedge_k\shift k^2$. 
Choose a basis $\{a_1,a_2\}$ for $C_2$.  With $b_i=(-1)^i\varsigma(a_i)$ for $i=1,2$,  
$a_3=a_1a_2$ and $b_3=a_1b_2$ the set $\boldsymbol{a}=\{1,a_i,b_i\}_{1\les i\les3}$ 
is a $k$-basis for $B$.  The non-zero products of elements of $\boldsymbol{a}$ are listed below:
  \begin{equation}
    \label{eq:B2}
a_1a_2=-a_2a_1=a_3
\quad\text{and}\quad 
a_1b_2=a_2b_1=b_1a_2=b_2a_1=b_3\,.
  \end{equation}

Let $\alpha\in(B^*)_{-2}$, respectively, $\beta\in(B^*)_{-3}$ be the $k$-linear 
map that sends $a_3$, respectively, $b_3$ to $1$, and 
the remaining elements of $\boldsymbol{a}$ to~$0$.  The map defined by
  \[
\pi\big(\varsigma^{-2}(x),\varsigma^{-3}(y)\big)=x\alpha-(-1)^iy\beta
  \]
for $x\in B_i$ and $y\in B_{i+1}$ is a morphism 
$\pi\col\shift^{-2}B\oplus\shift^{-3}B\to B^*$ of graded $B$-modules. 
Its image contains the basis of $B^*$ dual to $\boldsymbol{a}$, so $\pi$ is surjective.

Set $U=\Ker(\pi)$.  The surjectivity of $\pi$ implies $\rank_kU=7$.  Set
  \[
u_j=\big(\varsigma^{-2}(a_j),(-1)^j\varsigma^{-3}(b_j)\big)\,,
v_j=\big(\varsigma^{-2}(b_j),0\big)\,,
w=\big(0,\varsigma^{-3}(a_3)\big)\,,
  \] 
and $\boldsymbol{u}=\{u_j,v_j,w\}_{j=1,2,3}$.  It is easy to see that 
$\boldsymbol{u}$ is in $U$ and is linearly independent over $k$.
 Thus, $\boldsymbol{u}$ is a $k$-basis of $U$, so there is an isomorphism of vector spaces
  \[
\upsilon\col\shift^{-2}B_+\oplus\shift^{-1}k\xra{\,\cong\,} U
  \]
satisfying $\upsilon(a_j)=u_j$ and $\upsilon(b_j)=v_j$ for $j=1,2,3$, and $\upsilon(1)=w$.   
Simple calculations, using \eqref{eq:B2}, yield $\upsilon(bu)=b\upsilon(u)$ for all $b\in\boldsymbol{a}$
and $u\in\boldsymbol{u}$. This means that $\upsilon$ is $B$-linear, and so validates the 
exact sequence \eqref{eq:B1}.
   \end{Subsec}

\begin{Subsec}[Class $\mathbf{G}(r)$]
Formulas \eqref{eq:null2P} and \eqref{eq:null2I} give
 \begin{align*}
\frac1{\Po Bk}
&=1-rt^2-rt^3+ t^5\,.
  \\
 \frac{\ba B}{\Po Bk}
&=t^{-3}-rt^{-1}-r+t^2\,.
  \end{align*}
{From} formulas \eqref{eq:trivialP} and \eqref{eq:trivialI} we now obtain:
 \begin{align*}
\frac1{\Po Ak}
&=(1-rt^2-rt^3+t^5)-t\big((l+1-r)t+(l+n-r)t^2+(n-1)t^3\big)
  \\
&=(1+t)(1-t-lt^2-nt^3+t^4)\,.
  \\
\frac{\ba A}{\Po Ak}
&=(t^{-3}-rt^{-1}-r+t^2)+\big((l+1-r)t^{-1}+(l+n-r)t^{-2}+(n-1)t^{-3}\big)
 \\
&={t^{-3}(1+t)\big(n+(l-r)t-(r-1)t^2-t^3+t^4\big)}\,.
 \end{align*}
\end{Subsec}

\begin{Subsec}[Class $\mathbf{H}(p,q)$]  
When $p=0=q$ we have $A=k\ltimes W$, so 
\eqref{eq:trivialP} and \eqref{eq:syzygyC} give
  \begin{align*}
\frac 1{\Po Ak}
&=1-t\big((l+1)t+(l+n)t^2+nt^3\big)=(1+t)(1-t-lt^2-nt^3)\,.
  \\
\frac{\ba A}{\Po Ak}
&=(l+1)t^{-1}+(l+n)t^{-2}+nt^{-3}-t
=t^{-3}(1+t)(n+lt+t^2-t^3)\,.
 \end{align*}

When $(p,q)\ne(0,0)$, using  \eqref{eq:kunneth}, \eqref{eq:nullP},  and \eqref{eq:nullI} 
we get
 \begin{align*}
\frac 1{\Po Bk}
&=(1-pt^2-qt^3)\cdot(1-t^2)\,.
  \\
{\ba B}
&=\frac{qt^{-2}+pt^{-1}-t}{1-pt^2-qt^3}\cdot t^{-1}
=\frac{qt^{-3}+pt^{-2}-1}{1-pt^2-qt^3}\,.
 \end{align*}
{From} \eqref{eq:trivialP} and \eqref{eq:trivialI} we now obtain
 \begin{align*}
\frac1{\Po Ak}
&=(1-pt^2-qt^3)(1-t^2)-t\big((l-p)t+(l+n-p-q)t^2+(n-q)t^3\big)
  \\
&=(1+t)\big(1-t-lt^2-(n-p)t^3+qt^4\big)\,.
  \\
\frac{\ba A}{\Po Ak}
&=(qt^{-3}+pt^{-2}-1\big)(1-t^2)
  \\
&{\phantom{{}={}(qt^{-3}}}+\big((l-p)t^{-1}+(l+n-p-q)t^{-2}+(n-q)t^{-3}\big)
 \\
&=t^{-3}(1+t)\big(n+(l-q)t-pt^2-t^3+t^4\big)\,.
 \end{align*}
These formulas gives the desired expressions for $\Po Ak$ and $\ba A$.
\qedhere
\end{Subsec}
  \end{proof}

  \section{Classification}
     \label{S:Classification}

We significantly tighten the classification of rings $R$ of embedding codepth $3$,
recalled in \ref{table}, by proving that membership in each one of the classes 
described there imposes non-trivial restrictions on the numerical invariants 
of $R$.   Comparison with existing examples raises intriguing questions, discussed 
at the end of the section.

  \begin{theorem}
    \label{thm:class}
Let $(R,\fm,k)$ be a local ring with $\edim R-\depth R=c\le3$.

When $R$ is not Gorenstein the invariants from
\emph{\ref{invariants}} satisfy the following relations.



\medskip\noindent
\centerline{
 \begin{tabular}{l||l|r|l|l|l|l|l}
\emph{Class}
& $c$
& $h\le{\phantom{1}}$
& $l\ge$
& $n\ge$
& $p$
& $q$
& $r$
\\
\hline\hline
$\mathbf{S}$
& $2$
& $1$
& ${\phantom{\max\{}}2-h$
& ${\phantom{\max\{}}0=n$
& $0$
& $0$
& $0$
  \\
$\mathbf{T}$
& $3$
& $1$
& ${\phantom{\max\{}}3-h$
& ${\phantom{\max\{}}2$
& $3$
& $0$
& $0$
  \\
$\mathbf{B}$
& $3$
& $1$
& ${\phantom{\max\{}}4-h$
& ${\phantom{\max\{}}2-h$
& $1$
& $1$
& $2$
\\
$\mathbf{G}(r)$
& $3$
& $1$
& $\max\{4-h,r+1\}$
& ${\phantom{\max\{}}2-h$ 
& $0$
& $1$
& $r$
\\
$\mathbf{H}(p,q)$
& $3$
& $2$
& $\max\{3-h,p,q+1,2\}$
& $\max\{2-h,p-1,q,1\}$
& $p$
& $q$ 
& $q$
\\
  \end{tabular}
  }
  \end{theorem}

The notation used in the theorem remains in force for the rest of the section.

  \begin{remark}
   \label{tabulated}
The entries in the columns for $c$, $p$, $q$, and $r$ are read off directly 
from the description of the graded algebra $A$ in \ref{table}.  
  \end{remark}
  
Some numerical equalities determine the structure of the ring $R$.

  \begin{corollary}
    \label{cor:class}
Assume $c=3$ and $R$ is not complete intersection.

The following conditions then are equivalent.
\begin{enumerate}[\quad\rm(i)]
  \item
$l=q+1$.
  \item
$l=p$ and $n=q$.
  \item
$R$ is in $\mathbf{H}(p,q)$ with $n=p-1$.
 \item
$\widehat R\cong P/(J+zR)$, where $(P,\fp,k)$ is a regular local ring, $J$ an
ideal of $P$ with $J\subseteq\fp^2$ and $\rank_k(J/\fp J)=l\ge2$, and $z$ a 
$P/J$-regular element in $\fp^2$.
  \end{enumerate}
  \end{corollary}
  
When $l$ or $n$ is small the theorem is complemented by more precise results.

  \begin{subsec}
     \label{smalll}
Let $R$ be a local ring with $c=3$.

  \begin{subsubsec}
     \label{3rel}
If $l=2$, then by \cite[Proof of 7.2]{Av:small} one of the following cases occurs: 
  \begin{enumerate}[\quad\rm(a)]
  \item
$h=0$ and $R$ is in $\mathbf{C}(3)$.
  \item
$h=1$ and $R$ is in $\mathbf{H}(2,1)$ with $n=1$, or in $\mathbf{H}(0,0)$ or 
$\mathbf{H}(1,0)$ with $n\ge1$, or in $\mathbf{H}(2,0)$ or $\mathbf{T}$ with $n\ge2$.
  \item
$h=2$ and $R$ is in $\mathbf{H}(0,0)$ with $n\ge1$.
  \end{enumerate}
  \end{subsubsec}

  \begin{subsubsec}
     \label{4rel}
If $l=3$ and $h=0$, then by \cite[Proof of Thm.\,2]{Av:aci} $R$ is in one of the classes:  
  \begin{enumerate}[\quad\rm(a)]
  \item
$\mathbf{H}(3,2)$ with $n=2$.
  \item
$\mathbf{T}$ with odd $n\ge3$.
  \item
$\mathbf{H}(3,0)$ with even $n\ge4$.
  \end{enumerate}
  \end{subsubsec}

  \begin{subsubsec}
     \label{brown}
If $l\ge4$, $h=0$, $n=2$, and $p>0$, then by \cite[4.5]{Br} $R$ is in one of the classes:
  \begin{enumerate}[\quad\rm(a)]
  \item
$\mathbf{B}$ with even $l$.
  \item
$\mathbf{H}(1,2)$ with odd $l$.
   \end{enumerate}
  \end{subsubsec}
    \end{subsec}

We start the proof of the theorem with some general considerations.

  \begin{lemma}
    \label{lem:class}
Write $\widehat R$ in the form $P/I$, with $P$ regular and $\dim P=e$;
see \emph{\ref{cst}}.  

When $R$ is not complete intersection the following assertions hold.
  \begin{enumerate}[\quad\rm(1)]
  \item
$l\ge c-h=e-\dim R\ge1$ hold. 
  \item
$l=1$ implies $c=2$ and $h=1$; furthermore, $I=(wx,wy)$ where $w$ is an 
$\widehat R$-regular element and $x,y$ is an $\widehat R$-regular sequence.
 \item
$h\le2$ holds; furthermore, $h=2$ implies that $R$ is in $\mathbf{H}(0,0)$.  
  \item
If $R$ is not Gorenstein and $c=3$, then $n\ge 2-h$. 
  \end{enumerate}
  \end{lemma}

  \begin{proof}
(1) Formula \eqref{eq:codim} gives $l+1>c-h=e-\dim R\ge1$.

(2)  When $l=1$ the ideal $I$ is minimally generated by two elements, 
see \eqref{eq:codim}; say, $I=(u,v)$.  As the regular local ring $P$ is 
factorial, we have $u=wx$ and $v=wy$ with relatively prime $x,y$.  The sequence 
$x,y$ is regular, so $\operatorname{pd}_P\widehat R=2$.  As $\widehat R$ is not 
complete intersection, the element $w$ is non-zero and not invertible, so $h=1$.  

(3)  {From} (1) we see that $h\le c-1\le2$ holds, and equality implies $e-\dim R=1$.
The ring $R$ then is Golod by \cite[5.2.5]{Av:barca}, and so it is in $\mathbf{H}(0,0)$ 
by \ref{golod}.

(4) For any maximal $R$-regular sequence $\bx$ standard results, see \cite[1.6.16]{BH}, give
  \[
A_3\cong((\bx):\fm)/(\bx)\cong
\Hom{R/(\bx)}k{R/(\bx)}\cong \ext dRkR\ne0\,,
  \]
so $n=\mu^d_R\ge1$.  Now recall that $R$ is Gorenstein if and only if $h=0$ and 
$\mu^d_R=1$.
  \end{proof}

It is proved in \cite{AKM} that for several classes of local rings $R$,
including those of embedding codepth at most $3$, there is a 
complete intersection ring $Q$ and a Golod homomorphism 
$Q\to R$.  This was used to show that for every finite module $M$ 
over such a ring, $\Po RM$ represents a rational function 
with fixed denominator.

In the proof of the next lemma we turn the tables:  By applying the formulas for 
$\Po Qk$ and $\Po Rk$ from Theorem \ref{thm:series} we express the  
Betti numbers $\beta^Q_i(R)$ in terms of the numerical invariants of $R$, 
defined in \ref{invariants}, then use information on the asymptotic behavior
of Betti numbers over complete intersections, obtained in \cite{AGP}.

  \begin{lemma}
    \label{lem:alt}
Set $\tau_R=1$ for $R$ in $\mathbf{T}$ and $\tau_R=0$ otherwise.

If $c=3$ and $R$ is not complete intersection, then the following dichotomy holds:
  \begin{enumerate}[\rm\quad(a)]
 \item
$l\ge q+2$ and  $n\ge p-\tau_R$, or
 \item
$l= q+1$ and $n= p-1-\tau_R$.
  \end{enumerate}
  \end{lemma}

  \begin{proof}
Parts (1) and (2) of Lemma \ref{lem:class} imply $l\ge2$ and $h\le2$.  
If $h=2$, then Lemma \ref{lem:class}(3) shows that $R$ is in 
$\mathbf{H}(0,0)$.  When $R$ is in $\mathbf{H}(0,0)$ the first pair
holds and the second fails. Until the end of the proof we  assume 
$h\le1$ and $p+q\ge1$.  

We choose an isomorphism $\widehat R\cong P/I$, as in \ref{cst}.  
By taking a close look at some arguments in \cite{AKM}, we set out 
to show next that $I$ contains a regular sequence $x,y$, such that for 
$Q=P/(x,y)$ the induced map $Q\to \widehat R$ is a Golod homomorphism. 

For $R$ in $\mathbf{T}$ such a sequence is found in the proof of \cite[6.1]{AKM}.  
It is also shown there that if $R$ is in $\mathbf{G}(r)$, $\mathbf{B}$, or $\mathbf{H}(p,q)$ 
with $p+q\ge1$, then for some $x\in I$ and $\ov P=P/xP$ the map $\ov P\to R$ is Golod.  
As the ideal $\ov I=I/xP$ of $\ov P$ has positive height, we can choose a minimal 
generator $y$ of $I$ so that its image in $\ov R$ is regular.  The natural map from 
$Q=\ov P/y\ov P$ to $R$ is Golod by \cite[5.13]{AKM}.  
By the definition of Golod homomorphisms, see Levin \cite{Le:Gol}, the following
equality then holds:
  \begin{equation}
    \label{eq:golodhom}
\Po Q{\widehat R}=\frac1t\cdot\left(1+t-{\Po Qk}\cdot\frac1{\Po Rk}\right)\,.
  \end{equation}

Inspecting the tabulated values of $p$ and $q$, see Remark \ref{tabulated}, we note that 
the various forms of $\Po Rk$ listed in Theorem \ref{thm:series} admit an
uniform expression, namely,
  \begin{equation}
    \label{eq:tau}
\Po Rk=\frac{(1+t)^{e-1}}{1-t-lt^2-(n-p)t^3+qt^4-\tau_R t^5}\,.
  \end{equation}
Now $Q$ is in $\mathbf{C}(2)$, so $\Po Qk$ is given by Theorem \ref{thm:series}, hence 
\eqref{eq:tau} and \eqref{eq:golodhom} yield
  \begin{align*}
(1-t)\cdot{\Po Q{\widehat R}}
&=\frac{1-t}t\bigg(1+t-\frac{(1+t)^{e-2}}{(1-t)^2}\cdot\frac{1-t-lt^2-(n-p)t^3+qt^4-\tau_R t^5}{(1+t)^{e-1}}\bigg)
  \\
&=\frac{1-t}{t(1-t)^2(1+t)}\big((1-t^2)^2-(1-t-lt^2-(n-p)t^3+qt^4-\tau_R t^5)\big)
  \\
&=\frac1{t(1-t^2)}\big(t+(l-2)t^2+(n-p)t^3-(q-1)t^4+\tau_R t^5)\big)
  \\
&=\frac{1+(n-p)t^2+\tau_R t^4}{1-t^2}+\frac{(l-2)t-(q-1)t^3}{1-t^2}
  \\
&=1+(l-2)t+(n+1-p)t^{2}
  \\
&\phantom{{}={}1+(l-2)t}+\sum_{i=1}^\infty(l-1-q)t^{2i+1}+\sum_{i=1}^\infty(n+1-p+\tau_R)t^{2i+2}\,.
  \end{align*}

The composite equality of formal power produces numerical equalities
  \begin{equation}
    \label{eq:cigrowth}
  \begin{alignedat}{2}
l-1-q&=\beta^Q_{2i+1}(\widehat R)-\beta^Q_{2i}(\widehat R)
&\quad&\text{for all}\quad i\ge1\,,
  \\
n+1-p+\tau_R&=\beta^Q_{2i+2}(\widehat R)-\beta^Q_{2i+1}(\widehat R)
&\quad&\text{for all}\quad i\ge1\,.
  \end{alignedat}
   \end{equation}

The ring $Q$ being complete intersection, the sequence of Betti numbers of
each finite $Q$-module is eventually either strictly increasing or constant, 
see \cite[8.1]{AGP} or \cite[9.2.1(5)]{Av:barca}.  Thus, the left-hand sides of
the equalities in \eqref{eq:cigrowth} are either both positive or both equal to 
zero.  This is just a rewording of the desired conclusion.
   \end{proof}
   
The proof of the next result, with its use of a DG module structure on a minimal 
$P$-free resolution of a dualizing complex for $\widehat R$, presents independent interest.

  \begin{lemma}
    \label{lem:Gr}
If $c=3$ and $R$ is not Gorenstein, then $l\ge r+1$ holds.
  \end{lemma}

  \begin{proof}
There is nothing to prove for $R$ in $\mathbf{T}$, as then $r=0$; see Remark \ref{tabulated}.

By the same remark, rings in $\mathbf{B}$ have $p=q=1$ and $r=2$.  Case (b) in Lemma 
\ref{lem:alt} then cannot hold, as it implies $n=0$, and case (a) gives $l\ge3=r+1$.   

Rings $R$ in $\mathbf{H}(p,q)$ have $r=q$, see Remark \ref{tabulated}, 
and Lemma \ref{lem:alt} gives $l\ge q+1$.

For the rest of the proof we assume that $R$ is in $\mathbf{G}(r)$.  Thus, its Koszul 
homology algebra $A$ has the form $A=B\ltimes W$, where $B$ is a Poincar\'e duality 
$k$-algebra with 
  \[
\rank_kB_1=r=\rank_kB_2\,,
\quad\rank_kB_3=1\,,
\quad
B_1\cdot B_1=0
  \]
and $W$ is a graded $B$-module with $B_+W=0$.  For every graded $B$-module $N$,
set $N'=\Hom kN{\shift^3k}$ and endow this graded vector space with the natural 
$B$-module structure described in \ref{DG}.

Choose $\beta\in (B')_0$ with $\Ker(\beta)=B_{\les2}$.  As $B$ has Poincar\'e duality, 
the homomorphism of left graded $B$-modules $\alpha\col B\to B'$ with $\alpha(1)=\beta$ is 
bijective; thus, 
  \begin{equation}
    \label{eq:Astar}
A'=B\beta\oplus  W'
  \quad\text{and}\quad
\alpha\col B\cong B\beta
  \end{equation}
as graded $B$-modules, where $B$ act on $A$-modules through the inclusion 
$B\subseteq A$.

As we may assume that $R$ is complete, we fix a Cohen presentation $R\cong P/I$,
a minimal resolution $F$ of $R$ as a $P$-module, a DG $P$-algebra structures on 
$F$; see~\ref{cst}.  Set $F'=\Hom PF{\shift^3P}$ and turn $F'$ into a DG $F$-module,
as in~\ref{DG}.  

Using \eqref{eq:DGchain} to identify the graded algebras $F\otimes_Pk$ and 
$A$, we get isomorphisms 
  \[
F'\otimes_Pk\cong\Hom PF{k\otimes_P\shift^3P}\cong\Hom k{F\otimes_Pk}{\shift^3k}=A'
  \]
of graded $A$-modules.  Choose $\xi\in F'_{0}$, so that these maps
send $\xi\otimes1$ to $(\beta,0)\in A'$; see \eqref{eq:Astar}.  The 
morphism $\phi\col F\to F'$ of left DG $F$-modules with $\phi(1)=\xi$  satisfies
  \begin{equation}
    \label{eq:phi}
(\phi\otimes_Pk)|_{B}=\alpha
  \quad\text{and}\quad
(\phi\otimes_Pk)|_{W}=0\,.
  \end{equation}

Let $Y$ denote the mapping cone of $\phi$.  We have $\HH_i(F)=0$ for $i\ge1$ 
by choice, and $\HH_i(F')=\ext{3-i}PRP=0$ for $i\ge 2$ because $h\le1$ holds
by Lemma \ref{lem:class}(3).  The exact sequences  
$\HH_{i-1}(F)\to\HH_i(Y)\to\HH_i(F')$ now yield $\HH_i(Y)=0$ for $i\ge2$.  

Note that $Y$ is a bounded complex of finite free $P$-modules.  If $y\in Y_i$ is an 
element with $\dd(y)=z\notin\fp Y_{i-1}$, then form a subcomplex of $Y$ as follows:
  \[
Z=\quad 0\to Py\xra{\dd|_{Py}}Pz\to0
  \]
Since $Z$ is contractible and splits off as a direct summand of $Y$, the natural 
morphism $Y\to Y/Z$ is a homotopy equivalence.   Iteration produces
a homotopy equivalence $Y\to X$, where $X$ is a bounded complex of 
finite free $P$-modules satisfying
  \begin{alignat}{2}
    \label{eq:Xmin}
\dd(X)&\subseteq\fp X\,,
  \\
    \label{eq:X}
\HH_i(X)&\cong\HH_i(Y)=0 &\quad&\text{for}\quad i\ge2\,,
  \\
    \label{eq:Xk}
X_i\otimes_Pk&\cong\HH_i(Y\otimes_Pk)
&\quad&\text{for}\quad i\in\BZ\,.
  \end{alignat}
  
The construction of $Y$ gives an isomorphism of complexes of $k$-vector spaces
  \[
Y\otimes_Pk\cong
\quad
\left\{\begin{gathered}
\xymatrixcolsep{2.1pc}
\xymatrixrowsep{-.2pc}
\xymatrix{
0 
& W_3
& W_2
& W_1
  \\
& \oplus
& \oplus
& \oplus
  \\
& B_3
\ar@{->}[ddr]^{\alpha_3}
& B_2
\ar@{->}[ddr]^{\alpha_2}
& B_1
\ar@{->}[ddr]^{\alpha_1}
& B_0
\ar@{->}[ddr]^{\alpha_0}
\\
& {\phantom{\oplus}}
& \oplus
& \oplus
  \\
& {\phantom{B_3\beta}}
& B_3\beta
& B_2\beta
& B_1\beta
& B_0\beta
\\
& {\phantom{\oplus}}
& {\phantom{\oplus}}
& \oplus
& \oplus
& \oplus
  \\
&&& W_2'
& W_1'
& W_0'
& 0
  }
\end{gathered}
\right.
  \]
where in view of \eqref{eq:Astar} and \eqref{eq:phi} all maps not represented by arrows are 
equal to zero and each $\alpha_i$ is bijective.  Now \eqref{eq:Xk} yields isomorphisms 
of vector spaces
  \begin{equation}
    \label{eq:W}
X_i\otimes_Pk\cong \begin{cases}
W'_{i} &\text{for}\quad i=0,1\,,\\
W'_{2}\oplus W_1&\text{for}\quad i=2\,,\\
W_{i-1} &\text{for}\quad i=3,4\,.
  \end{cases}
    \end{equation}
The following equalities come from the definitions of $W'$ and $W$, \eqref{eq:switch} and \eqref{eq:m}:
  \begin{alignat*}{3}
\rank_kW_0'&=\rank_kW_3 &&=\rank_kA_3-\rank_kB_3&&=n-1\,,
  \\
\rank_kW_1'&=\rank_kW_2 &&=\rank_kA_2-\rank_kB_2&&=l+n-r \,,
  \\
\rank_kW_2'&=\rank_kW_1&&=\rank_kA_1-\rank_kB_1&&=l+1-r\,.
  \end{alignat*}
As a result, we now know that the complex $X$ has the following form: 
  \[
X=\quad
0\xra{\phantom{\dd_1}} P^{n-1}\xra{\,\dd_4\,} P^{l+n-r}\xra{\,\dd_3\,} 
P^{2(l+1-r)}\xra{\,\dd_2\,}P^{l+n-r}\xra{\,\dd_1\,} P^{n-1}\xra{\phantom{\dd_1}} 0
  \]

The inclusion $B_1\subseteq A_1$ yield $r\le l+1$.  We finish the proof by showing that
if $r=l+1$, then $R$ is Gorenstein, and that $r=l$ is not possible.

If $r=l+1$, then $X_2=0$, so the map $\dd_4\col P^{n-1}\to P^{n-1}$ 
is bijective.  In view of \eqref{eq:Xmin}, this forces $n=1$, hence $X=0$. {From}
\eqref{eq:W} we get $W=0$, so $A$ has Poincar\'e duality, and hence $R$ is Gorenstein
by \cite[Thm.]{AG}; see also \cite[3.4.5]{BH}.

Assume now $r=l$.  By  \eqref{eq:Xmin}  and \eqref{eq:X}, $\dd_2(X_2)$ has a minimal free 
resolution 
  \[
0\to P^{n-1}\to P^{n}\to P^{2}\to 0
  \] 
Since $\dd_2(X_2)$ is torsion-free, it is isomorphic to an ideal of $P$ minimally 
generated by two elements.  Such ideals have projective dimension one, see
Lemma \ref{lem:class}(2), hence $n=1$.  As $W_1\ne0$, the algebra $A$ does 
not have Poincar\'e duality, so the ring $R$ is not Gorenstein; see \ref{Gorl}. Thus, 
parts (3) and (4) of
Lemma \ref{lem:class} imply $h=1$; that is, $\dim R=d+1$.  A result of Foxby 
\cite[3.7]{Fo2} (for equicharacteristic $R$) and Roberts \cite{Ro} (in general) 
now gives $\mu^{d+1}_R\ge2$.  This inequality, the exact sequence
  \[
0\to\ext{d+1}RkR\to  A_{2}\xra{\delta_2}\Hom k{A_1}{A_3}
  \]
of \cite[Prop.\,1]{AG}, see also \cite[3.4.6]{BH}, the equality \eqref{eq:m}, 
and our assumption yield 
  \[
2\le\mu^{d+1}_R=l+n-r=1\,.
  \]
We have obtained a contradiction, and this finishes the proof of the lemma.
  \end{proof}

  \begin{proof}[Proof of Theorem \emph{\ref{thm:class}}]
For the values of $c$, $p$, $q$, and $r$, see Remark \ref{tabulated}.

Lemma \ref{lem:class}(3) yields $h\le2$, with strict inequality when $R$ is not in $\mathbf{H}(0,0)$. 

For $l$ and $n$ we argue one class at a time.

\begin{Subsec}[Class $\mathbf{S}$]
We have $l\ge 2-h$ by Lemma \ref{lem:class}(1) and $n=0$ by \eqref{eq:ab}.
  \end{Subsec}

  \begin{Subsec}[Class $\mathbf{T}$]
Lemma \ref{lem:class}(1) gives $l\ge 3-h$.  Rings in $\mathbf{T}$ have $q=0$, 
so case (b) in Lemma \ref{lem:alt} implies $l=1$; since $c=3$, this is ruled out 
by Lemma \ref{lem:class}(2).  Thus, the inequalities (a) of Lemma \ref{lem:alt}
hold, and they give $n\ge2$.
  \end{Subsec}

  \begin{Subsec}[Class $\mathbf{B}$]
Lemma \ref{lem:class}(4) gives $n\ge 2-h\ge1$, while  
Lemma \ref{lem:Gr} yields $l\ge r+1=3$.  By \ref{4rel}, the class $\mathbf{B}$ contains 
no ring with $h=0$ and $l=3$, so $l\ge4-h$ holds.  
  \end{Subsec}

  \begin{Subsec}[Class $\mathbf{G}(r)$]
Here $p=0$, so case (b) in Lemma \ref{lem:alt} gives $n=-1$, which is absurd.  
Thus, case (a) holds, whence $l\ge3$.  By \ref{4rel}, in $\mathbf{G}(r)$ there are 
no rings with $h=0$ and $l=3$, hence $l\ge4-h$ holds.  So does $l\ge r+1$, 
by Lemma~\ref{lem:Gr}.
  \end{Subsec}

  \begin{Subsec}[Class $\mathbf{H}(p,q)$]
Parts (1) and (3) of Lemma \ref{lem:class} give $l\ge\max\{3-h,2\}$. 

By definition, $A=(C\otimes_kD)\ltimes W$ with $C=k\ltimes(\shift k^p\oplus\shift^2k^q)$,
$D=k\ltimes\shift k$, and $C_+W=0=D_+W$.  The relations $A_1\supsetneq 
C_1\cong C_1\otimes_kD_1=A_1\cdot A_1$ imply  $l\ge p$, while $A_3\supseteq 
A_1\cdot A_2=C_2\otimes_kD_1\cong C_2$ yield $n\ge q$.  On the other hand, 
from Lemma \ref{lem:alt} we obtain the inequalities $l\ge q+1$ and $n\ge p-1$. 
 \qedhere
  \end{Subsec}
    \end{proof}

  \begin{proof}[Proof of Corollary \emph{\ref{cor:class}}]
When $l=q+1$ the values of $q$ and bounds for $l$ in 
Theorem \ref{thm:class} show that $R$ is in $\mathbf{H}(p,q)$. 
Lemma \ref{lem:alt} now gives $n=p-1$, so (i) implies (iii).

If (iii) holds, then we have a string $l\ge p=n+1\ge q+1= l$, where the inequalities 
come from Theorem \ref{thm:class}, the first equality is given by Lemma \ref{lem:alt},
and the second one holds by hypothesis.  We get $l=p$ and $n=q$, which is (ii).  

Assuming that (ii) holds, we see from Theorem \ref{thm:class} that $R$
is in $\mathbf{H}(p,q)$. The description of $A$ in \ref{table} then yields
$A\cong C\otimes_kD$ with $D=k\ltimes\shift k$.  In particular, if $a$ is a 
non-zero element in $1\otimes_kD_1$, then $A$ is free as a graded module 
over its subalgebra generated by $a$.  Now \cite[3.4]{Av:msri} shows that (iv) holds.

When (iv) holds $\Tor iP{P/J}{P/zP}=0$ for $i\ge1$, so we have isomorphisms 
  \begin{align*}
A&\cong\Tor{}P{P/(J+zP)}k
  \\
&\cong\Tor{}P{P/J}k\otimes_k\Tor{}P{P/zP}k
  \\
&\cong\Tor{}P{P/J}k\otimes_k(k\ltimes\shift k)
  \end{align*}
of graded $k$-algebras.  They imply $\Tor2P{P/J}k\otimes_k\shift k\cong A_3$
and  $\operatorname{pd}_P(P/J)=2$ the latter because $A_i=0$ holds for $i>3$ 
by \eqref{eq:ab}. We get a string of equalities
  \[
q=\rank_kA_3=\rank_k\Tor2P{P/J}k=\rank_k(J/\fp J)-1=l-1\,,
  \]
where the third one comes from \eqref{eq:m} and \eqref{eq:codim}.  Thus, (iv) implies (i).
  \end{proof}

To complete the classification of rings $R$ with $\dim R-\depth R=3$ along the lines 
of \ref{table} and the results in this section, one needs to determine for those rings
\emph{all} the restrictions satisfied by the invariants in \ref{invariants}.  This leads to:

  \begin{question}
Which sextuples $(h,l,n,p,q,r)$, allowed by Theorem \ref{thm:class}, Corollary 
\ref{cor:class}, or the results cited in \ref{smalll}, are realized by some local ring 
$R$ with $c=3$?
  \end{question}

The list of available answers is not long and runs as follows.

  \begin{subsec}
     \label{exist}
Let $(P,\fp,k)$ be a regular local with $\dim P=e\ge3$ and $x_1,\dots, x_e$ a minimal 
set of generators of $\fp$.  We describe rings $R=P/I$ with $c=3$ by specifying $I$.

  \begin{subsubsec}
    \label{preceding}
The rings admitted by \ref{Gorl}, \ref{3rel}, \ref{4rel}, and \ref{brown} are realized by 
ideals $I$ constructed in \cite[6.2]{BE2}, \cite[7.7]{Av:small}, \cite[Rem.\,(1), p.171]{Av:aci},   
and \cite[3.4, 3.6]{Br}, respectively.
   \end{subsubsec}

  \begin{subsubsec}
    \label{Hexist}
The following sextuples $(h,l,n,p,q,r)$ with $l=q+1$ are realized:
  \begin{enumerate}[\quad\rm(a)]
    \item
$(0,2,1,3,1,3)$ by $I=(x_1^2,x_2^2,x_3^2)$.
    \item
$(0,l,l-1,l,l-1,l-1)$ by $I=(x_1,x_2)^{l-1}+(x_3^2)$ for each $l\ge3$.
    \item
$(1,l,l-1,l,l-1,l-1)$ by $x_1(x_1,x_2)^{l-1}+(x_3^2)$ for each $l\ge2$.
  \end{enumerate}
There are no other sextuples with $l=q+1$, by Corollary \ref{cor:class}
and Lemma~\ref{lem:class}(3).
   \end{subsubsec}

   \begin{subsubsec}
Every sextuple $(2,l,n,0,0,0)$ with $l\ge2$ and $n\ge1$ is realized when
$k=\mathbb{C}$.

Indeed, for each such pair $(l,n)$ Weyman \cite{We1} shows that  
$P=\mathbb{C}[\![x_1,\dots,x_e]\!]$ contains an ideal $J$ with 
$\rank_k\Tor1P{P/J}k=l+1$ and $\rank_k\Tor3P{P/J}k=n$.  On the other 
hand, if $w$ is a $P$-regular element, then $P/J$ with $J=wI$ realizes the 
sextuple $(2,l,n,0,0,0)$, since for each $i\ge1$ there are isomorphisms of vector spaces
  \[
\Tor iP{P/I}k\cong\Tor{i-1}P{I}k\cong\Tor{i-1}P{wI}k\cong\Tor iP{P/J}k\,,
  \]
and Shamash \cite[Thm.\,(3), p.\,467]{Sh} shows that $P/wI$ is Golod; 
see also \cite[5.2.5]{Av:barca}.

There are no other sextuples with $h=2$, by Lemma \ref{lem:class}(3).
   \end{subsubsec}
  \end{subsec}

The only known examples in $\mathbf{G}(r)$ are the Gorenstein rings. 
We propose:

  \begin{conjecture}
If $R$ is in $\mathbf{G}(r)$ for some $r\ge2$, then $R$ is Gorenstein.
  \end{conjecture}

Lemma \ref{lem:Gr} is a first step towards a verification of this statement.  
If proved in full, it will eliminate an entire family from the classification in 
Theorem \ref{thm:class}.  

Another elusive class is $\mathbf{B}$, for which the only examples are those 
in \cite{Br}.  Rings in~$\mathbf{T}$ appear in several situations, and the families
$\mathbf{H}(p,q)$ seem to be ubiquitous.

\section{Bass numbers}
   \label{S:Bass numbers}

The following theorem is the third main result of this paper.  

  \begin{theorem}
    \label{thm:growth}
Let $(R,\fm,k)$ be a local ring, and set $e=\edim R$ and $d=\depth R$.

When $e-d\le3$ and $R$ is not Gorenstein there is real number 
$\gamma{\vphantom{\mu^{d}}}_R>1$, such that
  \begin{equation}
    \label{eq:big}
\mu^{d+i}_R\ge\gamma{\vphantom{\mu^{d}}}_R\,\mu^{d+i-1}_R
\quad\text{holds for every}\quad i\ge1\,,
   \end{equation} 
with two exceptions for $i=2$: If there exists an isomorphism
  \begin{align}
    \label{eq:exception2}
\widehat R&\cong P/(wx,wy)
\quad\text{or}
  \\
    \label{eq:exception3}
\widehat R&\cong P/(wx,wy,z)\,,
   \end{align} 
where $(P,\fp,k)$ is an $e$-dimensional regular local ring, $w$ a $P$-regular element, 
$x,y$ a $P$-regular sequence, and $z$ a $P/(wx,wy)$-regular element in $\fp^2$, then
  \[
\mu^{d+2}_{R}=\mu^{d+1}_{R}=2 \,.
  \]
  
In particular, when $R$ is Cohen-Macaulay the inequalities \eqref{eq:big} hold for all $i$.
 \end{theorem}

The theorem should be viewed in the context of a number of problems raised in recent 
publications, sometimes under the hypothesis that $R$ is Cohen-Macaulay.  We say 
that a sequence $(a_i)$ of real numbers is said to have \emph{strongly exponential 
growth} if $\beta^i\ge a_i\ge\alpha^i$ hold for all $i\gg0$ for some real numbers 
$\beta\ge\alpha>1$.  

%

\begin{problems}
  \label{qu}
Assume that $(R,\fm,k)$ is a non-Gorenstein local ring.
  \begin{enumerate}[\rm(1)]
  \item
Determine the number $\inf\{j\in\BZ\mid \mu^{d+i}_R>\mu^{d+i-1}_R\text{ for all }i\ge j\}$.
(See \cite[1.3]{CSV}.)
  \item
Does $\mu^{d+1}_R>\mu^d_R$ always hold?  
(See \cite[2.6]{JL}.)
  \item
Does $\mu^i_R\ge2$ hold for all $i>\dim R$? 
(See \cite[1.7]{CSV}.)
  \item
Does the sequence $(\mu^i_R)$ have strongly exponential 
growth?
(See \cite[p.\,647]{JL}.)
  \end{enumerate}
  \end{problems}

All of these questions are open in general.  Here is a list of the known answers:

 \begin{remark}
    \label{rem:answers}
Assume that $(R,\fm,k)$ is not Gorenstein.
  \begin{enumerate}[\rm(1)]
  \item
An inequality $\mu^{d+i}_R>\mu^{d+i-1}_R$ holds for $i\ge1$ in the following cases: 
  \begin{enumerate}[\quad\rm(a)]
  \item
$\fm^3=0$; see \cite[5.1]{CSV}.
  \item
$R\cong Q/(0:\fq)$ for some Gorenstein local ring $(Q,\fq,k)$; see \cite[6.2]{CSV}.
  \item
$R\cong S\times_kT$ with $S\ne k\ne T$, \emph{except}  when $S$ is a 
discrete valuation ring, and either $\edim T=1>\dim T$ or 
$\edim T=2=\dim T$; see \cite[3.3]{CSV}.
  \item
$R$~is Golod, \emph{except} when $e-d=2$ and $\mu^d_R=1$; see \cite[2.4]{CSV}.
  \end{enumerate}
In addition, $\mu^{d+i}_R>\mu^{d+i-1}_R$ is known to hold in the following cases:
  \begin{enumerate}[\quad\rm(a)]
  \item[\rm(e)]
for $i\ge3$ if $R$ is among the exceptions in (c) and (d); see \cite[2.5, 3.2]{CSV}.
  \item[\rm(f)]
for $i\gg0$ if $R$  is Cohen-Macaulay with $e-d\le3$; see \cite[1.1]{JL}.
  \end{enumerate}
\item
holds when $R$ is Cohen-Macaulay and is generically Gorenstein;
see \cite[2.3]{JL}.
\item
holds when $R$ is a domain, see \cite[p.\,67]{Ro:mont}, 
or is Cohen-Macaulay; see \cite[1.6]{CSV}.
\item

holds in  cases (a) through (f) of (1); see the references given above.
  \end{enumerate}
  \end{remark}

For rings with $e-d\le3$, Theorem \ref{thm:growth} provides sharp answers to all the questions 
in~\ref{qu}.  The next remark shows that the theorem also implies \ref{rem:answers}(1)(e).

 \begin{remark}
    \label{rem:exceptions}
Let $S$ be a discrete valuation ring and $(T,\ft,k)$ a local ring.

If $\edim T=1$ and $\ft^s=0\ne\ft^{s-1}$, then $R=S\times_kT$ satisfies $\widehat R\cong P/(wx, w^s)$, 
where $(P,\fp,k)$ is a regular local ring and $\{w,x\}$ is a minimal generating set for $\fp$. 

If $T$ is regular of dimension $2$,  then $R=S\times_kT$ satisfies $\widehat R\cong P/(wx, wy)$, where 
$(P,\fp,k)$ is a regular local ring and $\{w,x,y\}$ is a minimal generating set for $\fp$. 
  \end{remark}

In preparation for the proof of Theorem \ref{thm:growth}, we establish a technical 
result where the hypotheses are made on the Bass series of $R$ and the 
Poincar\'e series of $k$, not on the ring $R$ itself.   The argument relies on 
general properties of $\Po Rk$.
  
  \begin{lemma}
    \label{lem:comparison}
Let $(R,\fm,k)$ be a local ring and let $d$, $e$, $l$, $m$, and $p$ be as in \emph{\ref{invariants}}.

Assume there exist polynomials $f(t)$ and $g(t)$ in $\BZ[t]$,
such that
  \begin{equation}
    \label{eq:rational}
\Po Rk=\frac{(1+t)^{e-1}}{g(t)}
  \quad\text{and}\quad
\ba R=t^d\cdot\frac{f(t)}{g(t)}\,.
  \end{equation}
  \begin{enumerate}[\rm(1)]
   \item
If $\sum_{i=0}^\infty a_it^i$ is the Taylor expansion of $(f(t)-g(t))/(1-t^2)$, 
then
 \begin{align*}
\mu^{d}_R&=a_0+1\,,
  \\
\mu^{d+1}_R-\mu^{d}_R&=a_1-1\,,
  \\
\mu^{d+2}_R-\mu^{d+1}_R&=a_2+(l-1)a_0\,.
\intertext{In case $l\ge1$ and $a_i$ is non-negative for $i\ge1$ the following inequalities hold:}
\mu^{d+i}_R-\mu^{d+i-1}_R& \ge a_{i}+(l-1)a_{i-2} \ge a_i
  \quad\text{for}\quad i\ge2\,.
  \end{align*}
    \item
If $\sum_{i=0}^\infty b_it^i$ is the Taylor expansion of ${f(t)}(1+t^3)^s/{(1-t^2)^{2}}$, where
$s$ is an integer satisfying $0\le s\le m-p$, then
 \begin{align*}
\mu^{d}_R&=b_0 \,,
  \\
\mu^{d+1}_R-\mu^{d}_R&=b_1 \,,
  \\
\mu^{d+2}_R-\mu^{d+1}_R&=b_2+(l-2)b_0\,.
\intertext{In case $l\ge2$ and $b_i$ is non-negative for $i\ge1$ the following inequalities hold:}
\mu^{d+i}_R-\mu^{d+i-1}_R& \ge b_{i}+(l-2)b_{i-2} \ge b_i
\quad\text{for}\quad i\ge2\,.
  \end{align*}
 \end{enumerate}
   \end{lemma}
 
  \begin{Remark}
 One has $m-p=\rank_kA_2-\rank_k(A_1)^2=\rank_k(A_2/(A_1)^2)\ge0$.
   \end{Remark}

  \begin{proof}
Recall that the Poincar\'e series of $k$ can be written as a product
  \begin{equation}
    \label{eq:product}
\Po Rk=\frac{(1+t)^e(1+t^3)^{m-p}}{(1-t^2)^{l+1}}\cdot\frac{\prod_{i=2}^\infty(1+t^{2i+1})^{\varepsilon_{2i+1}}}
{\prod_{i=1}^\infty(1-t^{2i+2})^{\varepsilon_{2i+2}}}
  \end{equation}
with non-negative integers $\varepsilon_j\ge0$; see \cite[3.1.2(ii), 3.1.3]{GL} or \cite[7.1.4, 7.1.5]{Av:barca}.  

To compare consecutive Bass numbers, we will use the identity
  \begin{equation}
    \label{eq:difference}
\sum_{i\in\BZ}(\mu^{d+i}_R-\mu^{d+i-1}_R)\,t^i 
=(1-t)\frac{\ba R}{t^d}\,.
  \end{equation}

(1)  In view of \eqref{eq:product}, for $j\ge0$ there exist non-negative integers $c_j$, such that
  \begin{equation*}
\Po Rk=\frac{(1+t)^e}{(1-t^2)^{l+1}}\bigg(1+\sum_{j=3}^\infty c_jt^j\bigg)\,.
  \end{equation*}

Formulas \eqref{eq:difference} and \eqref{eq:rational} give equalities
  \begin{align*}
\sum_{i\in\BZ}(\mu^{d+i}_R-\mu^{d+i-1}_R)\,t^i 
&=\left(1+\frac{f(t)-g(t)}{g(t)}\right)(1-t) 
\\
&=1-t+\frac{f(t)-g(t)}{1-t^2}\frac{1}{(1-t^2)^{l-1}}\bigg(1+\sum_{j=3}^\infty c_jt^j\bigg)
\\
&=1-t+\bigg(\sum_{i=0}^\infty a_it^i\bigg)\bigg(1+(l-1)t^2+\sum_{j=3}^\infty d_jt^j\bigg)
 \end{align*}
with $d_j\ge0$ for $j\ge3$.  They yield $\mu^d_R=a_0+1$ and the expressions for 
$\mu^{d+i}_R-\mu^{d+i-1}_R$ when $i=1,2$.  In case $l\ge1$, and $a_i\ge0$ holds for $i\ge1$, 
we get the following relations, where  $\succcurlyeq$ denotes a coefficientwise inequality of 
formal power series
  \begin{align*}
\sum_{i\in\BZ}(\mu^{d+i}_R-\mu^{d+i-1}_R)\,t^i 
&\succcurlyeq 1-t+\bigg(\sum_{i=0}^\infty a_it^i\bigg)\big(1+(l-1)t^2\big)
  \\
&=a_0+(a_1-1)t+\sum_{i=2}^\infty\big(a_i+(l-1)a_{i-2}\big)t^{i}\,.
 \end{align*}
They imply the desired lower bounds for $\mu^{d+i}_R-\mu^{d+i-1}_R$ when $i\ge2$.

(2)  In view of \eqref{eq:product}, we can write $\Po Rk$ in the form 
  \begin{equation*}
\Po Rk=\frac{(1+t)^e(1+t^3)^s}{(1-t^2)^{l+1}}\bigg(1+\sum_{j=3}^\infty c_jt^j\bigg)
  \end{equation*}
with non-negative integers $c_j$.  Formulas \eqref{eq:difference} and \eqref{eq:rational} give  equalities
 \begin{align*}
\sum_{i\in\BZ}(\mu^{d+i}_R-\mu^{d+i-1}_R)\,t^i 
&=\frac{f(t)(1+t^3)^s}{(1-t^2)^2}\frac{1}{(1-t^2)^{l-2}}\bigg(1+\sum_{j=3}^\infty c_jt^j\bigg)
\\
&=\bigg(\sum_{j=0}^\infty b_jt^j\bigg)\bigg(1+(l-2)t^2+\sum_{j=3}^\infty d_jt^j\bigg)
  \end{align*}
with $d_j\ge0$ for $j\ge3$.  They yield $\mu^d=b_0$, and the expressions 
for $\mu^{d+i}_R-\mu^{d+i-1}_R$ when $i=1,2$.  When $l\ge2$, and $b_i\ge0$ holds 
for $i\ge1$, we also have
  \begin{align*}
\sum_{i\in\BZ}(\mu^{d+i}_R-\mu^{d+i-1}_R)\,t^i 
&\succcurlyeq \bigg(\sum_{i=0}^\infty b_it^i\bigg)\big(1+(l-2)t^2\big)
  \\
&=b_0+b_1t+\sum_{i=2}^\infty\big(b_i+(l-2)b_{i-2}\big)t^{i}\,.
 \end{align*}
The desired lower bounds for $\mu^{d+i}_R-\mu^{d+i-1}_R$ when $i\ge2$ follow from here.
  \end{proof}
 
The next lemma is the major step towards the proof of Theorem \ref{thm:growth}.
 
  \begin{lemma}
    \label{lem:growth}
If $(R,\fm,k)$ is a non-Gorenstein local ring with $e-d\le3$, then
  \[
\mu^{d+i}_R\ge\mu^{d+i-1}_R+1
  \quad\text{holds for}\quad i\ge1\,,
  \]
unless $i=2$ and $\widehat R$ is described by \eqref{eq:exception2} 
or \eqref{eq:exception3}, and then
  \[
\mu^{d+2}_R=\mu^{d+1}_R=2\,.
  \]
  \end{lemma}

  \begin{proof}
Once again, there are several different cases to consider.

  \begin{Subsec}[Class $\mathbf{S}$]
Theorem \ref{thm:series} gives $(1-t-lt^2)\ba R=t^d(l+t-t^2)$, hence 
  \begin{alignat*}{2}
\mu^{d}_R&=l\,,
  \\
\mu^{d+1}_R-\mu^{d}_R&=1\,,
  \\
\mu^{d+2}_R-\mu^{d+1}_R&=l^2-1\,,
  \\
\mu^{d+i}_R-\mu^{d+i-1}_R&=l\mu^{d+i-2}_R\ge 2
&&\quad\text{for}\quad i\ge3\,.
  \end{alignat*}

We get $\mu^{d+i}_R\ge\mu^{d+i-1}_R+1$ for all $i\ge1$, except when $i=1$ and $l=1$, and then 
$\mu^{d+2}_R=\mu^{d+1}_R=2$.  Furthermore, $l=1$ implies \eqref{eq:exception2} by 
Lemma \ref{lem:class}(2).
  \end{Subsec}

For the rest of the proof we assume $c=3$ and let $f(t)$ and $g(t)$ be 
 the polynomials from Theorem \ref{thm:series}, satisfying $\Po Rk=(1+t)^{e-1}/f(t)$ 
and $\ba S=t^df(t)/g(t)$. 

  \begin{Subsec}[Class $\mathbf{T}$]
The value of $f(t)$ from Theorem \ref{thm:series} provides the first equality below:
  \begin{align*}
\frac{f(t)}{(1-t^2)^2}
&=\frac{n+lt-2t^2-t^3+t^4}{(1-t^2)^2}
\\
&=(n-2t^2+t^4)\sum_{j=0}^\infty(j+1)t^{2j}+(lt-t^3)\sum_{j=0}^\infty(j+1)t^{2j}
  \\
&=n+\sum_{j=0}^\infty\big((l-1)j+l\big)t^{2j+1}+\sum_{j=1}^\infty(n-1)(j+1)t^{2j}
  \end{align*} 
  
Theorem \ref{thm:class} gives $l,n\ge 2$, so
Lemma \ref{lem:comparison}(2) applies with $s=0$ and yields
  \begin{alignat*}{2}
\mu^{d+1}_R-\mu^{d}_R&=l-1\ge1\,,
   \\
\mu^{d+2j}_R-\mu^{d+2j-1}_R&\ge(n-1)(j+1)\ge2  &\quad\text{for}\quad j&\ge1\,,
  \\
\mu^{d+2j+1}_R-\mu^{d+2j}_R&\ge(l-1)j+l\ge 2 &\quad\text{for}\quad j&\ge1\,.
  \end{alignat*}
  \end{Subsec}

  \begin{Subsec}[Classes $\mathbf{B}$, $\mathbf{G}(r)$, and $\mathbf{H}(0,0)$]
Theorem \ref{thm:class} provides uniform expressions, 
  \begin{align*}
f(t)&=n+(l-r)t-(r-1)t^2+(p-1)t^3+qt^4 \quad\text{and}
\\
g(t)&=1-t-lt^2-(n-p)t^3+qt^4\,,
 \end{align*}
for the polynomials that appear in Theorem \ref{thm:series}.  Using them, we obtain
  \begin{align*}
\frac{f(t)-g(t)}{1-t^2}
&=\frac{(n-1)(1+t^3)+(l+1-r)(t+t^2)}{1-t^2}
\\
&=(n-1)+(l+1-r)t+(l+n-r)\bigg(\sum_{j=2}^\infty t^{j}\bigg)\,.
 \end{align*}
 
Lemma \ref{lem:Gr} gives $l+n-r\ge l+1-r\ge2$, so the series above has non-negative
coefficients.  Thus, Lemma \ref{lem:comparison}(1) applies and yields 
  \begin{alignat*}{2}
\mu^{d+1}_R-\mu^{d}_R&=l-r\ge1\,,
  \\
\mu^{d+i}_R-\mu^{d+i-1}_R&\ge l+n-r\ge 2
&&\quad\text{for}\quad i\ge2\,.
  \end{alignat*}
  \end{Subsec}

 \begin{Subsec}[Class $\mathbf{H}(p,q)$ with $p+q\ge1$]  Theorem \ref {thm:series} gives
\begin{align*}
\frac{f(t)(1+t^3)}{(1-t^2)^2}
&=\frac{n+(l-q)t-pt^2-t^3+t^4}{(1-t^2)^2}(1+t^3)
  \\
&=n+(l-q)t+(2n-p)t^2+\sum_{i=3}^\infty b_it^{i}\,,
 \end{align*}
where for $i\ge3$ the numbers $b_i$ are defined by the formulas
  \begin{alignat*}{2}
b_{2j+1}&=(l-q+n-p)j+l+p-q-2 &\quad\text{for}\quad j&\ge1\,,
  \\
b_{2j}&=(l+n-p-q)j-l+n+q+1 &\quad\text{for}\quad j&\ge2\,.
 \end{alignat*}

By Theorem \ref{thm:class}, we have $l-q\ge1$, $n-p\ge-1$, and $n\ge 1$, hence
  \begin{alignat*}{2}
2n-p&=n+(n-p)\ge n-1\ge0 \,,
  \\
b_{2j+1}\ge b_3&=n-2+2(l-q)\ge n\ge 1 &\quad\text{for}\quad j&\ge1\,,
  \\
b_{2j}\ge b_4&=n+2(n-p)+(l-q)\ge n\ge1 &\quad\text{for}\quad j&\ge2\,.
   \end{alignat*}
Thus, Lemma \ref{lem:comparison}(2) applies with $s=1$.  With the preceding 
inequalities, it gives
  \begin{align*}
\mu^{d+1}_R-\mu^{d}_R&=l-q\ge1\,,
  \\
\mu^{d+2}_R-\mu^{d+1}_R&=2n-p +(l-2)n=ln-p\ge l(n-1)\ge0\,,
  \\
\mu^{d+i}_R-\mu^{d+i-1}_R&\ge b_i\ge 1 \quad\text{for}\quad i\ge3\,.
   \end{align*}
We conclude that $\mu^{d+i}_R\ge\mu^{d+i-1}_R+1$ holds for all $i\ge1$, except 
when $i=2$ and $ln-p=l(n-1)=0$.  To finish the proof, we unravel this special case.
 
The last two equalities force $n=1$ and $l=p$.  Now Theorem \ref{thm:class} gives the 
inequalities in the string $2=n+1\ge p=l\ge2$, whence $l=p=n+1=2$.  
Thus, we have shown that condition (iii) in Corollary \ref{cor:class} holds with $n=p-1=1$.
{From} condition (ii) in that corollary we get $q=n=1$, so the formulas above yield
  \[
\mu^{d+2}_R=\mu^{d+1}_R=\mu^{d}_R+l-q=n+l-q=2\,.
  \]
On the other hand condition (iv) gives an isomorphism $\widehat R\cong P/(J+zR)$, 
where $(P,\fp,k)$ is a regular local ring, $J$ is an ideal of $P$ contained in $\fp^2$ 
and minimally generated by $2$ elements, and $z$ is an element of $\fp^2$ that
is regular on $P/J$.  Since $R$ is not complete intersection, neither is $P/J$, which
means that $J=(wx,wy)$ for some non-zero element $w$ in $\fp$ and 
$P$-regular sequence $x,y$.  Thus, \eqref{eq:exception3} holds.
 \qedhere
  \end{Subsec}
   \end{proof}

  \stepcounter{theorem}

    \begin{proof}[Proof of Theorem \emph{\ref{thm:growth}}]
We may assume that $R$ is complete.  A construction of Foxby, see \cite[3.10]{Fo2},
then yields a finite $R$-module $N$, such that
  \begin{equation}
    \label{eq:foxby}
\mu^{d+i}_R=\beta^{R}_{i}(N)\quad\text{for all}\quad i\ge \dim R-d\,.
  \end{equation}

By \cite[1.4 and 1.6]{Av:msri}, when $c\le3$ the Betti sequence of every finite 
$R$-module either has strongly exponential growth or is eventually constant.  
Since $R$ is not Gorenstein, Lemma \ref{lem:growth} rules out the second 
case for the module $N$ in \eqref{eq:foxby}.  Thus, $\beta^{R}_{i}(N)\ge\alpha^i$ 
holds for some real number $\alpha>1$ and all $i\gg0$.   

The series $\Po RN$ converges in a circle of radius $\rho>0$, see 
\cite[4.1.5]{Av:barca}.  As $\rho$ is equal to $\limsup_i\{1/\sqrt[i]{\beta^i_R(N)}\,\}$ 
we get $0<\rho\le1/\alpha<1$.  Fix a real number $\beta$ satisfying
  \begin{equation}
    \label{eq:beta}
1/\rho>\beta>1\,.
  \end{equation}
Sun \cite[1.2(c)]{Su} proved that there is an integer $f$, such that 
  \begin{equation}
    \label{eq:sun}
\beta_{i}^R(N)\ge\beta\beta_{i-1}^R(N)
\quad\text{holds for all}\quad i\ge f+1\,.
  \end{equation}
Set $j=\max\{3,\dim R-d,f\}$ and define real numbers $\gamma'$ and $\gamma''$ by the formulas
  \begin{align*}
\gamma'&=\min\left\{\beta\,,\mu^{d+1}_R/\mu^{d}_R\,,\min\{\mu^{d+i}_R/\mu^{d+i-1}_R\}_{3\les i\les j}\right\}\,,
  \\
\gamma''&=\min\left\{\gamma'\,,\mu^{d+2}_R/\mu^{d+1}_R\right\}\,.
  \end{align*}

In view of \eqref{eq:foxby} and \eqref{eq:sun}, the following inequalities then hold:
  \begin{equation*}
\mu^{d+i}_R\ge
  \begin{cases}
\gamma'\mu^{d+i-1} &\text{for $i=1$ and } i\ge3\,,\\
\gamma''\mu^{d+i-1}&\text{for }i\ge 1\,.
  \end{cases}
  \end{equation*}
{From} \eqref{eq:beta} and Lemma \ref{lem:growth} we see that $\gamma''>1$ 
holds unless $\widehat R$ satisfies \eqref{eq:exception2} or \eqref{eq:exception3}, else 
$\gamma'>1$  holds and $\mu^{d+2}_R=\mu^{d+1}_R=2$.  This is the desired result.
   \end{proof}
  
\appendix

\section{Graded algebras}
  \label{S:Graded algebras}

Here $k$ denotes a field and $B$ a graded $k$-algebra that is 
graded-commutative, has $B_0=k$ and $B_i=0$ for $i<0$, and 
$\rank_kB$ is finite; set  $B_+=B_{\ges1}$.  

In addition, $M$ and $N$ denote finitely generated graded $B$-modules; we set 
  \[
H_M(t)=\sum_{i\in\BZ}\rank_kM_i\,t^i\,.
  \]

We treat $B$ as a DG algebra and $M$, $N$ as DG $B$-modules, all 
with zero differentials.   As a consequence, $\Tor {i}BMN$ and $\ext {i}BMN$ 
are formed as in \ref{DGD}.  The $k$-spaces $\Tor iBMN$ and $\ext iBMN$ 
are finite for each $i\in\BZ$ and zero for $i\ll0$, so \emph{Poincar\'e series} 
$\Po BM$ and \emph{Bass series} $\baa BN$ are defined; see \eqref{eq:defP} and \eqref{eq:defI}.

Here we assemble a collection of such series, used in the body of the paper.  Their
computations rely on analogs of results concerning finite modules over local rings.  

\begin{subsec}
  \label{maximalP}
For the graded $B$-modules $\shift^sN$ and $N^*=\Hom kNk$, see \ref{DG}, one has
  \begin{alignat}{3}
   \label{eq:shift}
\Po B{\shift^sN}&=t^s\cdot{\Po B{N}}
  &\quad&\text{and}\quad
&\baa B{\shift^sN}(t)&=t^{-s}\cdot{\baa B{N}}\,.
   \\
   \label{eq:dual}
{\Po B{N^*}}&=\baa B{N}(t)\
  &\quad&\text{and}\quad
&\baa B{N^*}(t)&=\Po B{N}\,.
  \end{alignat}

An exact sequence $0\to N\to G\to M\to0$ with $G$ free and $N\subseteq B_+G$
yields
  \begin{equation}
   \label{eq:maximalP}
\Po BN=t^{-1}\cdot(\Po B{M}-H_{M/B_+M}(t))\,.
   \end{equation}
   
If $B=C\otimes_kD$ and $M=T\otimes_kU$, where $T$ is a graded $C$-module and
$U$ a graded $D$-module, then the K\"unneth Formula gives
  \begin{equation}
   \label{eq:kunneth}
\Po BM=\Po CT\cdot\Po DU
  \quad\text{and}\quad
\baa B{M}=\baa C{T}\cdot\baa D{U}\,.
   \end{equation}
  \end{subsec}

The formulas above suffice to compute $\Po Bk$ and/or $\ba B$ in
some simple cases.

\begin{example}
  \label{ex:trivial}
If $B=k\ltimes W$ for some graded $k$-vector space $W\ne0$, then:
   \begin{align}
  \label{eq:nullP}
{\Po Bk}
&=\frac{1}{1-t\cdot H_W(t)}\,.
  \\
   \label{eq:nullI}
\frac{\ba B}{\Po Bk}
&=H_W(t^{-1})-t\,.
  \end{align}
  
Indeed, $B_+^2=0$ implies $\Po B{B_+}=H_W(t)\cdot\Po Bk$, so \eqref{eq:nullP} 
follows from \eqref{eq:maximalP}.  As $B_+(B^*)=(B^*)_{\ges0}$, any lifting to 
$B^*$ of some basis of the $k$-space $(B_+)^*$ minimally generates $B^*$ 
over $B$.  Thus, there is an exact sequence of graded $B$-modules
  \[
0\to U\to B\otimes_k(B_+)^*\to B^*\to 0
  \]
with $U\subseteq B_+\otimes_k(B_{+})^*$, hence
$B_+U=0$.  {From} this and \eqref{eq:dual} we obtain
  \[
{\ba B}={\Po B{B^*}}=H_U(t)\cdot t\cdot\Po Bk+H_W(t^{-1})\,,
  \]
because $H_{(B_+)^*}(t)=H_W(t^{-1})$.  Since $H_{B}(t)=H_W(t)+1$, the 
sequence also gives
  \[
H_U(t)
=(H_W(t)+1)\cdot H_W(t^{-1})-(H_W(t^{-1})+1)
=H_W(t)\cdot H_W(t^{-1})-1\,,
  \]
because $H_{B^*}(t)=H_B(t^{-1})$.  {From} the last two formulas
and \eqref{eq:nullP}, we get
  \begin{align*}
\frac{\ba B}{\Po Bk}
&=t\cdot\big(H_W(t)\cdot H_W(t^{-1})-1\big)+H_W(t^{-1})\cdot\big(1-t\cdot H_W(t)\big)
=H_W(t^{-1})-t\,.
   \end{align*}
\end{example}

\begin{example}
  \label{ex:exterior}
If $B=\bigwedge_kV$, where $V_i=0$ for all even $i$, then there is an equality
  \begin{equation}
   \label{eq:exteriorP}
\Po B{k}=\prod_{i\in\BZ}\frac 1{(1-t^{i+1})^{\rank_kV_i}}\,.
  \end{equation}

Indeed, set $c=\rank_kV$.  When  $c=1$ the isomorphism 
$\bigwedge_k\shift^i k\cong k\ltimes\shift^i k$ and \eqref{eq:nullP} give the 
desired expression.  For $c\ge2$ it is obtained by induction, using the isomorphism 
$\bigwedge_k(V'\oplus V'')\cong \bigwedge_kV'\otimes_k\bigwedge_kV''$
and  \eqref{eq:kunneth}.
 \end{example}

\begin{example}
  \label{ex:dual}
When $B$ has Poincar\'e duality in degree $s$ there is an equality 
  \begin{equation}
   \label{eq:exteriorI}
\ba B=t^{-s}\,.
  \end{equation}

Indeed, the condition on $B$ means that the $B_i\to\Hom k{B_{s-i}}{B_s}$, induced 
by the products $B_i\times B_{s-i}\to B_s$, are bijective for all $i\in\BZ$.  This implies 
an isomorphism $B^*\cong\shift^{-s}B$ of graded $B$-modules, so \eqref{eq:shift} and
\eqref{eq:dual} give
  \[
\ba B=\Po B{B^*}=\Po B{\shift^{-s}B}=t^{-s}\cdot\Po BB=t^{-s}\,.
  \]
 \end{example}

The next result is an analog of a theorem of Gulliksen; see \cite[Thm.\,2]{Gu:gol}.  
The original proof, or the one for \cite[Cor.\,2]{He}, carries over essentially without changes.

\begin{subsec}
  \label{trivial}
If $B=C\ltimes W$ for some graded $k$-algebra $C$ and graded $C$-module $W$, then
  \begin{equation}
   \label{eq:trivialP}
\frac1{\Po Bk}=\frac1{\Po Ck}-t\cdot\frac{\Po CW}{\Po Ck}\,.
   \end{equation}
  \end{subsec}

\begin{example}
If $B=C\ltimes \shift^s(C^*)$ with $C=k\ltimes W$, then the following hold:
   \begin{alignat}{2}
  \label{eq:null2P}
{\Po Bk}
&=\frac1{1-t\cdot H_W(t)-t^{s+1}\cdot H_W(t^{-1})+t^{s+2}}\,.
  \\
  \label{eq:null2I}
\frac{\ba B}{\Po Bk}
&=\frac{1-t\cdot H_W(t)-t^{s+1}\cdot H_W(t^{-1})+t^{s+2}}{t^{s}}\,.
  \end{alignat}

Indeed, the isomorphism of graded 
$B$-modules $\shift^s(C^*)\cong(\shift^{-s}C)^*$ and \eqref{eq:dual} 
give $\Po C{\shift^s(C^*)}=\baa C{\shift^{-s}C}(t)=t^s\cdot \ba C$. Now 
\eqref{eq:trivialP}, \eqref{eq:nullP}, and \eqref{eq:nullI} yield
  \[
\frac{1}{\Po Bk}
=\frac{1}{\Po Ck}-t\cdot t^s\cdot\frac{\ba C}{\Po Ck}
=1-t\cdot H_W(t)-t^{s+1}\cdot H_W(t^{-1})+t^{s+2}\,.
  \]
Since $B$ has Poincar\'e duality in degree $s$, 
\eqref{eq:null2P} and \eqref{eq:exteriorI} imply \eqref{eq:null2I}.
  \end{example}

The following analog of a result of Lescot, see \cite[1.8(2)]{Ls}, can be proved
along the lines of the original argument, but subtle changes are
needed.  Instead of going into those details, we refer to  \cite{AI} for
a direct proof covering both cases. 

\begin{subsec}
  \label{prop:lescot}
If $B_+\ne0$, then the following equality holds: 
  \begin{equation}
   \label{eq:syzygyC}
\frac{\ba B}{\Po Bk}=\frac{\baa B{B_+}(t)}{\Po Bk}-t\,.
  \end{equation}
\end{subsec}

In the last two examples we adapt the arguments for \cite[3.2(1) and 1.9]{Ls}.

\begin{example}
  \label{prop:trivialI}
If $B=C\ltimes W$ for some graded $k$-algebra $C$ with $C_+\ne0$ and 
graded $C$-module $W$ with $C_+W=0$, then the following equality holds:
   \begin{equation}
  \label{eq:trivialI}
\frac{\ba B}{\Po Bk}
=\frac{\ba C}{\Po Ck}+H_W(t^{-1})\,.
  \end{equation}

Indeed, $C$ is an algebra retract of $B$.  The  proof of \cite[Thm.\,1]{He}
transfers \emph{verbatim} and gives $\Po BN/\Po Bk=\Po CN/\Po Ck$
for each graded $C$-module $N$, wiewed as a $B$-module via the natural 
homomorphism $B\to C$.   By \eqref{eq:dual}, this implies
that $\baa BN(t)/\Po Bk=\baa CN(t)/\Po Ck$ holds as well.  Since $B_+=C_+\oplus W$
as graded $B$-modules, using the preceding equality and \eqref{eq:syzygyC} 
(twice) we obtain
   \begin{align*}
\frac{\ba B}{\Po Bk}
=\frac{\baa B{C_+}(t)}{\Po Bk}+\frac{\baa B{W}(t)}{\Po Bk}-t
=\frac{\baa C{C_+}(t)}{\Po Ck}+H_W(t^{-1})-t
=\frac{\ba C}{\Po Ck}+H_W(t^{-1})\,.
  \end{align*}
  \end{example}

\begin{example}
  \label{cor:lescot}
If $B=E/E_{\ges s}$, where $E$ is a graded $k$-algebra that has 
Poincar\'e duality in degree $s$, then the following equality holds:
  \begin{equation}
   \label{eq:truncatedI}
\frac{\ba B}{\Po Bk}=t^{-s-1}\cdot\bigg(1-\frac{1}{\Po Bk}\bigg)-t\,.
  \end{equation}

Indeed, set $(-)'=\Hom k-{\shift^sk}$.  Applying the functor $(-)'$ to 
the exact sequence $0\to B_+\to B\to k\to 0$ we get $(B_+)'\cong 
B'/B'_{\ges s}$ as graded $B$-modules.  Since $E\cong E'$ 
as graded $E$-modules, $(-)'$ applied to $0\to \shift^sk\to E\to B\to 0$ 
gives $B'\cong E_+$, hence $B'/B'_{\ges s}\cong E_+/E_{\ges s}
\cong B_+$, and thus $B_+\cong(B_+)'\cong(\shift^{-s}B_+)^*$.

Now from formulas \eqref{eq:syzygyC}, \eqref{eq:dual}, and 
\eqref{eq:maximalP} we obtain
  \[
\frac{\ba B}{\Po Bk}+t
=\frac{\baa B{(\shift^{-s}B_+)^*}}{\Po Bk}
=\frac{\Po B{\shift^{-s}B_+}}{\Po Bk}
=t^{-s}\cdot t^{-1}\cdot\frac{\Po Bk-1}{\Po Bk}\,.
  \]
  \end{example}

All the labeled formulas in this appendix are used in computations in Section 
\ref{S:Series}.

  \section*{Acknowledgements}
I thank Lars Winther Christensen, Oana Veliche, and an anonymous 
referee for a number of corrections and remarks that have significantly 
cleaned up the text.

\end{document}